\documentclass[11pt]{ip-journal}

\usepackage{verbatim,amscd,epsfig,amsmath,amsthm,amstext,amssymb
}
\usepackage{graphicx,xcolor}
\usepackage{enumerate}
\usepackage[all]{xy}
\usepackage[hyperindex=true,plainpages=false,colorlinks=false,pdfpagelabels]{hyperref}
\usepackage{verbatim}

\theoremstyle{plain}
\newtheorem{The}{Theorem}
\newtheorem*{The*}{Theorem}
\newtheorem{Pro}{Proposition}
\newtheorem{Lem}{Lemma}

\newtheorem*{Cor*}{Corollary}
\newtheorem{Fac}{Observation}

\theoremstyle{definition}

\newtheorem{Def}{Definition}
\newtheorem{Rem}{Remark}

\newtheorem*{Rem*}{Remark}

\numberwithin{equation}{section}

\DeclareMathOperator{\End}{End}
\DeclareMathOperator{\Hom}{Hom}
\DeclareMathOperator{\Sym}{Sym}

\DeclareMathOperator{\im}{im}
\DeclareMathOperator{\diag}{diag}

\DeclareMathOperator{\dbar}{\bar\partial}
\DeclareMathOperator{\del}{\partial}
\DeclareMathOperator{\ord}{ord}
\DeclareMathOperator{\tr}{tr}

\DeclareMathOperator{\Res}{Res}
\DeclareMathOperator{\supp}{supp}
\DeclareMathOperator{\dist}{dist}

\newcommand{\R}{\mathbb{R}}
\newcommand{\C}{\mathbb{C}}
\newcommand{\N}{\mathbb{N}}
\newcommand{\Z}{\mathbb{Z}}
\renewcommand{\P}{\mathbb{P}}

\newcommand{\todo}[1] {\textbf{\color{red}{#1}}}

\begin{document}

\author{Sebastian Heller}
\author{Charles Ouyang}
\author{Franz Pedit}
\title{Higgs bundles and SYZ geometry}


\maketitle
\begin{abstract}
Using non-Abelian Hodge theory for parabolic Higgs bundles, 
we construct infinitely many non-congruent  hyperbolic affine spheres modeled on a thrice-punctured sphere with monodromy in ${\bf SL}_3(\Z)$. These give rise to non-isometric semi-flat Calabi--Yau metrics on special Lagrangian torus bundles over an open ball in $\R^{3}$  with a Y-vertex deleted, thereby answering a question raised by Loftin, Yau, and Zaslow in \cite{LYZ}, \cite{LYZerr}.
\end{abstract}

\section{Introduction}

One of the predictions of string theory, which describes particles by strings rather than by points, is mirror symmetry. It asserts an equivalence between 
type {IIA} string theory in a spacetime
$Y$ and type {IIB} string theory 
in a spacetime $\check{Y}$.  In this setting, spacetime is $10$-dimensional and to make contact with $4$-dimensional physics, one assumes that at least locally $Y=\R^4\times X$. Consistency of string theory then requires $X$ to be a Calabi--Yau 3-fold and whence mirror symmetry predicts that Calabi--Yau 3-folds come in mirror pairs $(X,\check{X})$. Calculations in the A-model involve the symplectic geometry of $X$ given by the K\"ahler form, whereas calculations in the B-model involve the complex structure of $\check{X}$. This dichotomy led to  a number of spectacular predictions and eventual proofs of deep results in pure mathematics, such as  
the number of rational curves of a given degree in a quintic 3-fold (\cite{Giv}, \cite{LLY}; see also \cite{COGP}).

With the introduction of Dirichlet branes, serving as endpoints of open strings, the correspondence between symplectic data on $X$ and holomorphic data on the mirror $\check{X}$ was formulated in differential geometric terms by Strominger, Yau, and Zaslow \cite{SYZ}.
In the A-model, a Dirichlet brane 
is a special Lagrangian submanifold 
together with 
a flat, unitary line bundle over it.
In the B-model, a Dirichlet brane is a complex submanifold 
 together with a holomorphic line bundle over it.
 Mirror symmetry then predicts a bijection between the moduli spaces of A-branes on $X$ and B-branes on $\check{X}$. 
 Since points with the trivial line bundle are B-branes in $\check{X}$, this bijection implies that $X$ be swept out by a 
 family of A-branes. The moduli space of special Lagrangians  $S\subset X$ has tangent space $H^1(S,\R)$ at $S$ and the flat unitary line bundles $L\to S$ are parametrized by their Abelian monodromies $\Hom(H_1(S,\Z),S^1)$ \cite{McLean}, \cite{Hi3}. Thus, equating dimensions, the first Betti number of $S$ is  $b_1(S)=3$ and $X\to B$ should be a  3-torus fibration over the $3$-dimensional moduli space $B$ of special Lagrangians in $X$.  Elaborating on these ideas and interchanging the roles of $X$ and $\check{X}$
leads to the SYZ-conjecture \cite{SYZ}:

If $(X,\check{X})$ is a mirror pair of Calabi--Yau 3-folds, then $\pi\colon X\to B$ and $\check{\pi}\colon \check{X}\to B$ are special Lagrangian 3-torus fibrations over the same 3-dimensional base $B$ whose fibers $\pi^{-1}(b)$ and $\check{\pi}^{-1}(b)$ are dual tori.  

The base $B$ inherits a {\em integral flat special affine structure}, that is, a torsion-free flat connection $\nabla$ with a parallel volume form $\det_B\in\Omega^3(B,\R)$ such that the monodromy of $\nabla$ takes values in  ${\bf SL}_3(\Z)$.  The simplest examples of such fibrations are the semi-flat ones, where the Calabi--Yau metric is invariant---and hence flat---along the fibers. The geometry of such a Calabi--Yau manifold $X$ is completely determined by a {\em Monge--Amp{\`e}re metric} $\nabla d\phi$ on the base $B$. The 
 {\em real K\"ahler potential} $\phi\colon \tilde{B}\to \R$ is a convex function on the universal cover $\tilde{B}$ with monodromy in the Abelian group of affine functions. The Ricci-flatness of the resulting K\"ahler metric on $X$ is given by  the Monge--Amp{\`e}re equation $\det_B \nabla d\phi=1$ for the positive definite Hessian $\nabla d\phi$.
 Moreover, as a manifold, $X=TB/\Lambda$ is constructed from the tangent bundle of the base $B$  after quotiening by a $\nabla$-parallel full rank lattice bundle $\Lambda\subset TB$.  

In order to obtain interesting examples, as compact semi-flat Calabi--Yau manifolds are flat tori, the  fibration $\pi\colon X\to B$ needs to degenerate along a singular set $\Gamma\subset B$ of  codimension at least two \cite{Joyce}. Besides point singularities and intervals---which have been studied in this context---Loftin, Yau, and Zaslow \cite{LYZ} construct examples of Monge--Amp{\`e}re metrics with certain prescribed asymptotics, assuming $\Gamma$ to be modeled on a single trivalent vertex, a ``Y-vertex'', in a $3$-ball.

Their construction makes use of classical centro-affine differential geometry---the study of hypersurfaces in $\R^{n+1}$ with symmetry group ${\bf SL}_{n+1}(\R)$---initiated by the Blaschke School \cite{Blaschke}, \cite{NomSas}.  A natural class of such hypersurfaces are the affine spheres, which come in three types: elliptic, hyperbolic, and parabolic, depending on whether their affine normal lines intersect in a point on the concave or convex side of the hypersurface, or are parallel. A parabolic affine sphere in $\R^{n+1}$  is obtained as a graph of a convex function $\phi$ on a domain $B\subset \R^n$ satisfying the Monge--Amp{\`e}re equation $\det_B \nabla d\phi=1$ for a torsion-free flat connection $\nabla$ on $TB$ preserving a determinant form $\det_B\in \Omega^n(B,\R)$. 
 In other words, constructing Monge--Amp{\`e}re metrics is the same as constructing parabolic affine spheres.

The analysis of  Monge--Amp{\`e}re equations in dimensions three and higher is difficult, let alone 
controlling the asymptotics and monodromy of a solution along a singular set. To overcome this, the authors \cite{LYZ}, \cite{LYZerr} utilize a method of  Baues and Cort{\'e}s \cite{BauC}, initiated in special cases by  Calabi \cite{Calabi},  to construct parabolic affine spheres as cones over elliptic or hyperbolic affine spheres. Whence, one can construct Monge--Amp{\`e}re metrics on $B$, a 3-ball minus a Y-vertex, by constructing elliptic or hyperbolic 2-dimensional affine spheres modeled on the thrice-punctured sphere\footnote{The complex structure on $\Sigma$ is understood to be that of $\P^1\setminus\{p_1,p_2,p_3\}$.} $\Sigma=S^2\setminus \{p_1,p_2,p_3\}$.  The integrability equation for the existence of an elliptic or hyperbolic 2-dimensional affine sphere is the {\em Tzitz{\'e}ica equation}
\begin{equation}\label{eq:Tzit-globalintro}
2\triangle_{g_0} u +2|Q|^2_{g_0} e^{-4u}-He^{2u}+1=0\,,
\end{equation}
whose solution $g=e^{2u} g_0$ gives the Blaschke metric of the affine sphere. Here $g_0$ denotes the hyperbolic background metric of curvature $-1$ on $\Sigma$ and $Q\in H^0(K_{\Sigma}^3)$ is the holomorphic Pick differential of the affine sphere. 
The sign of the constant $H$ distinguishes the elliptic, $H=-1$, from the hyperbolic, $H=1$, affine spheres and has significant impact on the analysis of the Tzitz{\'e}ica equation \eqref{eq:Tzit-globalintro}. Note that the hyperbolic metric $g_0$ satisfies  \eqref{eq:Tzit-globalintro} in the hyperbolic affine sphere case for $Q\equiv 0$.

The main content of  \cite{LYZ} is the construction, via non-linear elliptic analysis, of a solution to \eqref{eq:Tzit-globalintro} in the elliptic affine sphere case assuming the cubic differential $Q$ to be small with quadratic poles at the punctures $p_k$ and the metric $g=e^{2u} g_0$ to be asymptotic to a rotational metric near $p_k$. Such a solution gives rise to 
an elliptic affine sphere  immersion $f\colon \tilde{\Sigma}\to\R^3$ from the universal cover $\tilde{\Sigma}$ together with a surface group representation $\rho\colon \pi_1(\Sigma)\to{\bf SL}_3(\R)$ 
such that $\gamma^*f =\rho_{\gamma}f$ for $\gamma\in\pi_1(\Sigma)$ acting via deck transformations. 
Applying the coning method of \cite{BauC} 
then provides a flat special affine structure on $B$,  a 3-ball minus a Y-vertex, with monodromy $\rho$ and a convex function $\phi\colon B\to \R$ giving rise to  the Monge--Amp{\`e}re metric $\nabla d\phi$ on $B$. Finally, to construct the Calabi--Yau manifold $X=TB/\Lambda$  as a special Lagrangian torus fibration over $B$, the monodromy $\rho\colon \pi_1(B)\to{\bf SL}_3(\R)$ needs to be integral.   As of yet, it is unknown whether the monodromy of the elliptic affine sphere constructed in \cite{LYZ}  is integral and thus one cannot compare this Calabi--Yau manifold to examples with integral monodromy constructed by algebro-geometric methods \cite{CJL},
\cite{Gross}, \cite{GrossSiebert}. In other words, the metric constructed in \cite{LYZ} only provides an example of a semi-flat Calabi--Yau  fibered over a 3-ball minus a Y-vertex by special Lagrangian $3$-planes, rather than 3-tori.

It has been pointed out in \cite{LYZerr} that one can modify the construction in \cite{BauC} and obtain parabolic affine spheres, and thus Monge--Amp{\`e}re metrics, by coning hyperbolic affine spheres. In this situation, the analysis of the Tzitz{\'e}ica equation \eqref{eq:Tzit-globalintro} becomes much more tractable. The equation is a particular reduction of Hitchin's ${\bf SU}_3$ self-duality equation over a Riemann surface \cite{Hi}, in our case the thrice-punctured sphere $\Sigma$, rather than a reduction of the non-compact ${\bf SU}_{2,1}$ self-duality equation for the elliptic affine sphere case. 
We can thus avail ourselves to results in non-Abelian Hodge theory, which sets up  correspondences between solutions to Hitchin's self-duality equations, Higgs bundles, and representations of surface groups in the case of a compact structure group such as ${\bf SU}_3$. These correspondences turn out to be explicit enough to construct infinitely many non-isometric solutions with integral monodromy $\rho$ to \eqref{eq:Tzit-globalintro} in the hyperbolic affine sphere case. Therefore, we obtain infinitely many non-isometric semi-flat Calabi--Yau 3-folds fibered by special Lagrangian tori over a 3-ball minus a Y-vertex.

Our starting point is the observation that a hyperbolic affine sphere modeled on the thrice-punctured sphere $\Sigma=S^2\setminus \{p_1,p_2,p_3\}$, whose holomorphic cubic Pick differential $Q\in H^0(K_{\Sigma}^3)$ extends meromorphically with quadratic poles into the punctures, corresponds to a stable parabolic Higgs bundle $(W,\Phi)$ over the Riemann sphere $S^2$, where 
\begin{equation}\label{eq:parHiggs}
W=\mathcal{O}(-1)\oplus   \mathcal{O}\oplus\mathcal{O}(1)\quad\text{and}\quad 
\Phi=\begin{pmatrix} 0 &{\bf 1} & 0\\ 0 & 0 &{\bf 1} \\ Q & 0 & 0 \end{pmatrix}.
\end{equation}
If  we introduce the singularity divisor $\mathfrak{D}=p_1+p_2+p_3$, then its corresponding line bundle $\mathcal{O}(\mathfrak{D})=\mathcal{O}(3)$ and since the canonical bundle  $K=K_{S^2}=\mathcal{O}(-2)$, the entries of the Higgs field should be interpreted as follows: ${\bf 1} $ is the unique (up to scale) section in $K\mathcal{O}(-1)\mathcal{O}(\mathfrak{D})=\mathcal{O}$  
 and $Q$, as a cubic differential with quadratic poles at $p_k$, is a holomorphic section of $K^3\mathcal{O}(2\mathfrak{D})=\mathcal{O}$.\footnote{We often omit the tensor product symbol if there is no risk of confusion.} This is consistent with viewing $Q$ as a holomorphic 1-form with values in $\mathcal{O}(2)$. Therefore, the Higgs field $\Phi\in H^0(K{\bf sl}(W)\mathcal{O}(\mathfrak{D}))$ is a meromorphic, trace-free endomorphism-valued 1-form on $\Sigma$ with
simple poles at the punctures $p_k$---encoded by the simple poles of the section ${\bf 1} $---and maximal nilpotent residues
\begin{equation}\label{eq:resHiggs}
\Res_{p_k}\Phi= \begin{pmatrix} 0 &\Res_{p_k}{\bf 1}  & 0\\ 0 & 0 &\Res_{p_k}{\bf 1} \\ 0 & 0 & 0 \end{pmatrix}.
\end{equation}
Note that the moduli space of Higgs bundles \eqref{eq:parHiggs} is parametrized by meromorphic cubic differentials $Q$ with quadratic poles at the punctures $p_k\in S^2$ and therefore is a complex line $\C$. 

Applying Simpson's \cite{S} generalization of the non-Abelian Hodge correspondence to parabolic bundles, we obtain to a Higgs bundle \eqref{eq:parHiggs} a unique Hermitian metric $h$ on $W_{|\Sigma}$ such that
$\Phi$ satisfies the self-duality equation 
\begin{equation}\label{eq:selfdual}
F^{D^h}+[\Phi\wedge \Phi^{\dagger_h}]=0\,
\end{equation}
over $\Sigma$ with $D^h$ the Chern connection. Due to the special form of our Higgs bundle \eqref{eq:parHiggs},  the Hermitian metric $h$ has to respect the decomposition $W=\mathcal{O}(-1)\oplus   \mathcal{O}\oplus\mathcal{O}(1)$, that is, has to be diagonal $h=h_{-1}\oplus h_0\oplus h_{1}$ with $h_{-1}=h_{1}^{-1}$.  Since $\mathcal{O}(1)=K\mathcal{O}(\mathfrak{D})$, the construction of the Higgs bundle \eqref{eq:parHiggs} from the hyperbolic affine sphere implies that $h_1$ defines a Riemannian metric $g$ on $\Sigma$ solving the Tzitz{\'e}ica equation \eqref{eq:Tzit-globalintro}. Moreover, the simple pole structure of the Higgs field $\Phi$ at the punctures $p_k$ implies that the Blaschke metric $g$ has bounded distance near the punctures $p_k\in S^2$ to the cusp metric $g_0$, the hyperbolic metric on the punctured disk which solves \eqref{eq:Tzit-globalintro} for $Q\equiv 0$. In particular, this metric has the same asymptotic behavior as the metric constructed in \cite{LYZerr}, \cite{L}.  The Blaschke metric $g$ and the cubic differential $Q=\det_W \Phi$ determine a unique hyperbolic affine sphere immersion $f\colon \tilde{\Sigma}\to\R^3$ on the universal cover $\tilde{\Sigma}$, which is equivariant with respect to a representation $\rho\colon \pi_1(\Sigma)\to{\bf SL}_3(\R)$.  This representation is of course the monodromy representation of the flat connection $D^h+\Phi+\Phi^{\dagger_h}$. Thus, we need to understand for which cubic  differentials $Q\in H^0(K_{\Sigma}^3)$, with quadratic poles at the punctures $p_k$, this representation is integral, that is, takes values in ${\bf SL}_3(\Z)$ possibly after  a conjugation.

It is known \cite{S} that for Higgs bundles \eqref{eq:parHiggs}  the local monodromies $\rho_k$ around the punctures $p_k$ have minimal polynomial $(\lambda-1)^3$. 
Let $\mathcal{M}_{B}^{ps}$ denote the relative character variety, that is, the moduli space of conjugacy classes 
of completely reducible representations $\rho\colon \pi_1(\Sigma)\to{\bf SL}_3(\C)$, whose generators have characteristic
 polynomial  $(\lambda-1)^3$. Based on a description of Lawton \cite{Law}, we show that $\mathcal{M}_{B}^{ps}$  can be realized via a character map $\mathcal{X}$ as  a cubic hypersurface $\mathcal{F}\subset \C^3$, whose only singularity corresponds to the trivial representation.
 All the smooth points  are irreducible representations. The variety of  real points  $\mathcal{F}(\R)=\mathcal{F}\cap \R^3$ corresponds to conjugacy classes of real representations $\mathcal{M}_{B}^{ps}(\R)$. The latter
has exactly two connected components: $\mathcal{C}_1$, containing the trivial representation, and the Hitchin component $\mathcal{C}_2$, containing the uniformization representation \cite{Hi2}. In particular,  
$\mathcal{C}_2$ is smooth and we show that the non-Abelian Hodge correspondence \cite{S} restricts to  a homeomorphism between the moduli space
of Higgs bundles of type \eqref{eq:parHiggs} and the Hitchin component $\mathcal{C}_2\subset \mathcal{M}_{B}^{ps}(\R)$ in the real relative representation variety. Via this correspondence, 
the Higgs bundle \eqref{eq:parHiggs} with $Q\equiv 0$ corresponds to the uniformization representation.  

At last we come to the question of integral representations. A necessary 
(and possibly sufficient)
condition for a representation $\rho\colon \pi_1(\Sigma)\to{\bf SL}_3(\R)$ to be integral is that its image under the character map $\mathcal{X}_{\rho}\in\mathcal{F}(\Z)$ is contained in the integer points $\mathcal{F}(\Z)=\mathcal{F}\cap \Z^3$ of the character variety. 
In other words, we have to solve a Diophantine problem. Utilizing the fact that the character variety $\mathcal{F}\subset \C^3$ 
admits an explicit birational parameterization, we construct an infinite sequence of non-congruent integral representations in the Hitchin component  $\mathcal{C}_2$ and  also in the component of the identity representation $\mathcal{C}_1$. This proves the following
\begin{The*}
There exists an infinite family of non-isometric  semi-flat
Calabi--Yau metrics on $\pi\colon X/\Lambda\to B$ fibered by special Lagrangian $3$-tori, where $B$ is an open $3$-ball with a Y-vertex deleted.
\end{The*}

The paper navigates between aspects of classical differential geometry and the theory of Higgs bundles, whose practitioners traditionally do not strongly interact. With this in mind, we have attempted to make the exposition largely self-contained and accessible to readers from both communities. This approach leads to passages in the paper of a more expository nature.

\subsection*{Acknowledgements} 
This project was initiated during a research visit of the first and third author
to the Tata Institute of Fundamental Research Mumbai. We thank Indranil Biswas for valuable discussions and TIFR for its hospitality and excellent research conditions.
The first author was supported by the  {\em Deutsche Forschungsgemeinschaft} within the priority program {\em Geometry at Infinity}. The second author acknowledges support from the National Science Foundation through grant DMS-2202832. We thank Max Alekseyev \cite{Al}  for providing an argument for the existence of infinitely many integer solutions of a certain
Diophantine equation. Finally, we thank the referees for their careful reading of the original manuscript and thoughtful comments.

\section{Affine spheres, Higgs bundles, and K{\"a}hler metrics}
\subsection{Affine hypersurface immersions}
We recall some of the results relating affine spheres with flat affine structures and Higgs bundles in a language suitable for the development of  this paper. There are many accounts of the various relationships between these topics, for instance, the classical works of Blaschke \cite{Blaschke}, Calabi \cite{Calabi}, Cheng and Yau \cite{CY}, and more recently Loftin, Yau, and Zaslow \cite{LYZ}, Labourie \cite{La}, and
Loftin \cite{Lsurvey}.
Affine differential geometry studies the geometry of oriented hypersurface immersions $f\colon M^n\to \R^{n+1}$ under the symmetry group of special affine transformations 
${\bf SL}_{n+1}(\R)\ltimes\R^{n+1}$ with respect to a constant volume form $\det\in \Lambda^{n+1}( {\R^{n+1}})^{*}$ on $\R^{n+1}$.  The immersion $f$ is called {\em non-degenerate} if the second fundamental form $\overline{d^2 f}\colon TM\times TM\to \underline{\R}^{n+1}/TM$ is non-degenerate as a symmetric bilinear bundle map, where $\underline{\R}^{n+1}=M\times
\R^{n+1}$ denotes the trivial bundle.  In this case, there exists a unique transverse line subbundle $L\subset \underline{\R}^{n+1}$, that is $\underline{\R}^{n+1}=TM\oplus L$,
and a unique up to sign trivializing section $\xi\in\Gamma(L)$ satisfying
\begin{equation}\label{eq:affine-normal}
\nabla^L\xi=(d\xi)^L=0\quad\text{and}\quad {\det}_g=\det(df,\dots,df,\xi)\,.
\end{equation}
Here ${\det}_g\in\Omega^{n}(M,\R)$ denotes the volume form of the pseudo-Riemannian metric $g=\overline{d^2 f}$ on $M$  obtained from the second fundamental form $\overline{d^2 f}$ by identifying $\underline{\R}^{n+1}/TM\cong L\cong  \underline{\R}$
via the trivializing section $\xi$. The line bundle $L$ is called the {\em affine normal line}, the trivializing section $\xi$ the {\em affine normal}, and the pseudo-Riemannian metric $g$ the {\em Blaschke metric}. In case $g$ is a definite metric, which means $f$ is convex, the sign of $\xi$ is chosen so that $g$ is positive definite and hence a Riemannian metric. The infinitesimal invariants of the hypersurface $f\colon M\to \R^{n+1}$ are obtained from the decomposition of the trivial $\underline{\R}^{n+1}$ connection
\begin{equation}\label{eq:decompose}
d=\begin{pmatrix} \nabla & S\xi^*\\ g\xi & \nabla^L\end{pmatrix}
\end{equation}
with respect to the splitting $\underline{\R}^{n+1}=TM\oplus L$. Here $\xi^*\in\Gamma(L^*)$ denotes the unique section satisfying $\langle \xi^*,\xi\rangle=1$.  The resulting torsion-free connection $\nabla$ on $TM$ is called the {\em Blaschke connection} and $S\in\Gamma(\End(TM))$ the {\em affine Weingarten map}. The flatness of $d$ provides  affine versions of the Gauss-Codazzi equations 
\begin{equation}\label{eq:GC}
R^{\nabla}=g\wedge S\,, \qquad g\circ S=S^*\circ g\,,\qquad  d^{\nabla} g=0\,,\quad d^{\nabla}S=0
\end{equation}
where the second identity expresses the self-adjointness of $S$ with respect to the Blaschke metric $g$. The Blaschke connection $\nabla$, even though torsion-free, is not metric. It relates to the Levi-Civita connection $\nabla^g$ of the Blaschke metric  via
\begin{equation}\label{eq:LeviCivita}
\nabla^g=\nabla+\frac{1}{2}g^{-1}\circ\nabla g\,.
\end{equation}
Due to the third equation in \eqref{eq:GC}, the {\em Pick} or {\em cubic form} 
\begin{equation}\label{eq:Pick}
C=-\frac{1}{2}\nabla g \in \Gamma(\Sym^3(TM,\R))
\end{equation}
is symmetric and---since both $\nabla$ and $\nabla^g$ preserve the volume form ${\det}_g$---has vanishing trace with respect to $g$.  The Pick form $C$  measures the deviation of the hypersurface $f$ from a quadric:  it vanishes identically if and only if $f$ is a quadratic hypersurface.

\subsection{Reconstruction and monodromy}\label{subsec:reconstruction}
For our purposes, it will be necessary to reconstruct a hypersurface  immersion $f\colon M\to\R^{n+1}$ from the geometric invariants $g,S$ and $C$.  Let $M$  be an $n$-dimensional manifold with a (pseudo)-Riemannian metric $g$, a self-adjoint bundle map $S\in\Gamma(\End(TM))$, and a symmetric form $C\in\Gamma(\Sym^3(TM,\R))$ which is trace-free with respect to $g$. Assume that these data fulfill the integrability conditions \eqref{eq:GC},  where the putative Blaschke connection $\nabla$ is given by \eqref{eq:LeviCivita}, \eqref{eq:Pick} as $\nabla=\nabla^g+g^{-1}\circ C$. Then the connection 
\begin{equation}\label{eq:FT}
d_V=\begin{pmatrix} \nabla & S\\ g & d_{\R}\end{pmatrix}
\end{equation}
on the real rank $n+1$ bundle $V=TM\oplus\underline{\R}$ is flat and the volume form ${\det}_V={\det}_g\wedge dt$ on $V$ is parallel. Thus,  $(V,d_V, {\det}_V)$ trivializes over the universal cover $\tilde{M}$ to $(\underline{\R}^{n+1}, d, \det)$ and fixing a base point $p_0\in M$, we have the holonomy representation 
\begin{equation}\label{eq:holonomy}
\rho\colon \pi_1(M,p_0)\to {\bf SL}_{n+1}(\R)
\end{equation}
of the flat connection $d_V$. Furthermore, since $\nabla$ is torsion-free, the inclusion $T\tilde{M}\subset \underline{\R}^{n+1}$ is a closed,  and thus exact,  $1$-form $df\colon T\tilde{M}\to \R^{n+1}$ for a non-degenerate immersion $f\colon \tilde{M}\to \R^{n+1}$. The holonomy representation \eqref{eq:holonomy} and the period representation $\tau\colon \pi_1(M,p_0)\to \R^{n+1}$ of $df$ combine to give the special affine representation 
$(\rho,\tau)\colon \pi_1(M,p_0)\to {\bf SL}_{n+1}(\R)\ltimes\R^{n+1}$, under which the hypersurface immersion $f$ is equivariant, that is,
\begin{equation}\label{eq:equiv}
\gamma^* f =\rho_{\gamma} f + \tau_{\gamma}\,.
\end{equation}
By construction, the Blaschke metric, affine Weingarten map, and Pick form are the data $g,S$ and $C$ from which we started. 

\subsection{Affine spheres}
For the construction of examples of mirror pairs of Calabi--Yau $3$-folds fibered by special Lagrangian 3-tori, we need to understand the geometry of the
 special class of {\em affine sphere} immersions $f\colon M\to\R^{n+1}$ which are convex, that is, the Blaschke metric $g$ is Riemannian. Those come in three types characterized by the condition that the affine normal lines $L_p$ for $p\in M$  meet in a point---which we may assume to be the origin in $\R^{n+1}$ not contained in the image of $f$---or are parallel (meet at infinity). Depending on whether the origin lies on the same or opposite side of the tangent planes as the hypersurface, the affine sphere is called {\em elliptic} or {\em hyperbolic}.  If the affine normal lines are parallel, the hypersurface is called a {\em parabolic} affine sphere. For the hyperbolic or elliptic affine spheres, the affine normal is $\xi=f$ or $\xi=-f$ so that the affine Weingarten map is $S=I$ or $S=-I$, respectively. The pictures to have in mind are one sheet of a hyperboloid or an ellipsoid. A parabolic affine sphere has constant affine normal $\xi$ and hence the affine  Weingarten map $S=0$ is trivial. The integrability equations \eqref{eq:GC} for affine spheres reduce to the first equation
\begin{equation}\label{eq:G}
R^{\nabla}=g\wedge S\,.
\end{equation}
In particular, parabolic affine spheres carry a flat special affine structure given by the Blaschke connection $\nabla$ and the parallel---with respect to $\nabla$ and also $\nabla^g$---volume form $\det_g$. By choosing a transverse  hyperplane $E\subset \R^{n+1}$ to $\R\xi$, a parabolic affine sphere can be (locally) parametrized as a graph over $M\subset E$ via 
\[
f(p)=p+\xi \phi(p)
\]
 for a smooth convex function $\phi\colon M\to \R$. From \eqref{eq:decompose}, we read off that the Blaschke metric $g$ of a parabolic affine sphere  is the Hessian $g=\nabla d\phi$  of $\phi$. The second condition in \eqref{eq:affine-normal},  characterizing the affine normal,  is equivalent to the Monge--Amp{\`e}re equation 

 \[
 {\det}_E \nabla d\phi =1
\]
for the graph function $\phi$ where ${\det}_E=\det(-,\xi)$. Thus, constructing Monge--Amp{\`e}re metrics on a domain is equivalent to constructing parabolic affine spheres.  This observation lies at the heart of the construction of semi-flat Calabi--Yau mirror pairs initiated in \cite{LYZ}. The authors construct  $3$-dimensional parabolic affine sphere metrics from  $2$-dimensional hyperbolic or elliptic affine sphere metrics via a coning method \cite{BauC}. In our setup, it is easier to work with an intrinsic version which also simplifies the proof of this construction.

Let $M$ be an $n$-dimensional manifold with Riemannian metric $g$ and a symmetric, trace-free form $C\in\Gamma(\Sym^3(TM,\R))$ satisfying the integrability condition \eqref{eq:G}
 with $S=H I$ and $\nabla=\nabla^g+g^{-1}\circ C$ where $H=\pm 1$.  As explained in Section~\ref{subsec:reconstruction}, the connection
\begin{equation}\label{eq:FTspheres}
d_V=\begin{pmatrix} \nabla & H I\\ g & d_{\R}\end{pmatrix}
\end{equation}
is flat on the rank $n+1$ bundle $V=TM\oplus\underline{\R}$  with parallel determinant ${\det_V}=\det_g\wedge dt$. Consider the ``coning''  
\[
B=M\times (0,1) 
\]
 of $M$ and let $\pi\colon B\to M$ be the projection $\pi(p,r)=p$. The tangent bundle of $B$ is  canonically identified with the pullback bundle $\pi^*V\to B$ via the bundle isomorphism
 \begin{equation}\label{eq:F-iso}
F\colon TB\to\pi^*V\,,\quad F(v,\mu \partial_r)=(rv,\mu H)\,.
\end{equation}
We denote by $\tilde{\nabla}=F^{-1}\circ \pi^*d_V\circ F$ and $\det_B=H F^*\pi^*{\det}_V$ the flat connection and $\tilde{\nabla}$-parallel volume form on $TB$ corresponding to $\pi^*d_V$ and $\pi^*{\det}_V$ under the isomorphism $F$. 
\begin{Lem}\label{lem:monodromy}
The connection $\tilde{\nabla}$ on $TB$ is torsion-free. Since the fundamental groups  of $M$ and $B$ are canonically isomorphic, the monodromy representation $\rho\colon \pi_1(B)\to {\bf SL}_{n+1}(\R)$ of $\tilde{\nabla}$ is the same as that  of the connection $d_V$ in  \eqref{eq:holonomy}. Hence, $B$ is a special affine flat $n+1$ dimensional  manifold with monodromy $\rho$. 

The function 
\[
\phi\colon B\to \R\,,\qquad \phi(r)=-H\int_0^r (1-H\rho^{n+1})^{\frac{1}{n+1}}\,d\rho\,,\quad H=\pm 1
\]
is convex and satisfies $\det_B \tilde{\nabla} d\phi=1$, thereby defining the Monge--Amp{\`e}re metric $\tilde{\nabla} d\phi$ on $B$. 
\end{Lem}
\begin{Rem}
This Lemma is an intrinsic version of results in \cite{BauC}, \cite{LYZ}, \cite{LYZerr} and in a special case attributed  to Calabi \cite{Calabi}. The authors consider the hyperbolic, $H=1$, respectively elliptic, $H=-1$, equivariant affine sphere immersion $f\colon \tilde{M}\to\R^{n+1}$ obtained from the flat connection $d_V$ as discussed in \eqref{subsec:reconstruction}. The $n+1$ dimensional manifold $B$ is the domain $B\subset\R^{n+1}$ obtained by coning the hypersurface $f$.  The graph of $\phi$ over $B$ is then shown to give a parabolic affine sphere in $\R^{n+2}$. This extrinsic view point somewhat obscures the monodromies of the affine spheres so constructed. 
\end{Rem}
\begin{proof}
Both assertions follow from calculating the covariant derivative $\tilde{\nabla}_{\tilde{X}}\tilde{Y}$ for vector fields $\tilde{X}=(X,a\,\partial_r)$ and $\tilde{Y}=(Y,b\,\partial_r)$ on $B$, where $X,Y\in\Gamma(TM)$ are vector fields on $M$ and $a,b\in \R$  are constants. Using the explicit form \eqref{eq:FTspheres} of the connection $d_V$, we obtain
\begin{equation*}
\begin{split}
F(\tilde{\nabla}_{\tilde{X}}\tilde{Y}) & =  (\pi^* d_V)_{\tilde{X}} F(\tilde{Y})= (\pi^* d_V)_{\tilde{X}} (rY,bH)\\ & 
=r(\pi^* d_V)_{\tilde{X}} (Y,0)+ dr(\tilde{X})(Y,0)+(\pi^* d_V)_{\tilde{X}}(0,bH) \\ &= 
 r (d_V)_{X}(Y,0)+ a(Y,0) + bH(d_V)_X (0,1)\\ &= 
 (r\nabla_X Y +aY + bX , r g(X,Y))\\&= F(\tfrac{aY+bX}{r}+\nabla_X Y, rHg(X,Y)\del_r)\,,
\end{split}
\end{equation*}
and therefore 
\begin{equation}\label{eq:nablatilde}
\tilde{\nabla}_{\tilde{X}}\tilde{Y}=(\tfrac{aY+bX}{r}+\nabla_X Y, rHg(X,Y)\del_r)\,.
\end{equation}
The latter  is symmetric in $\tilde{X}$, $\tilde{Y}$ and since $\nabla$ has no torsion, the connection $\tilde{\nabla}$ is torsion-free.

The Monge--Amp{\`e}re metric  $\tilde{\nabla} d\phi$ calculates to 
\[
\tilde{\nabla} d\phi=-H \tilde{\nabla} ((1-Hr^{n+1})^{\tfrac{1}{n+1}} dr)= (1-Hr^{n+1})^{\tfrac{-n}{n+1}}r^n dr^2-H (1-Hr^{n+1})^{\tfrac{1}{n+1}}\tilde{\nabla} dr\,.
\]
With $\tilde{X},\tilde{Y}$ as above and \eqref{eq:nablatilde}, we obtain
\[
(\tilde{\nabla}_{\tilde{X}} dr)_{\tilde{Y}}= \tilde{X}\cdot dr(\tilde{Y})-dr(\tilde{\nabla}_{\tilde{X}}\tilde{Y})= -rH\pi^*g(\tilde{X},\tilde{Y})
\]
and therefore
\begin{equation}\label{eq:MAmetric}
\tilde{\nabla} d\phi=(1-Hr^{n+1})^{\tfrac{-n}{n+1}}r^n dr^2+r(1-Hr^{n+1})^{\tfrac{1}{n+1}}\pi^*g\,.
\end{equation}
Since $r\in(0,1)$, this metric is positive definite on $B=M\times (0,1)$. Furthermore,
\[
{\det}_B=H F^*\pi^*({\det}_g\wedge dt)= H F^* {\det}_{\pi^*g}\wedge dt= r^n {\det}_{\pi^*g}\wedge dr\,,
\]
which implies that  $\tilde{\nabla}d\phi$ satisfies the Monge--Amp{\`e}re equation $\det_B \tilde{\nabla}d\phi=1$. 
\end{proof}

 \subsection{Two dimensional affine spheres and Higgs bundles}
The integrability condition for affine spheres \eqref{eq:G}, when expressed in terms of the Levi-Civita connection \eqref{eq:LeviCivita} and the Pick form \eqref{eq:Pick}, splits into  skew and self-adjoint components with respect to the Blaschke metric
\begin{equation}\label{eq:GCspheres}
R^g= H g\wedge I-g^{-1}\circ C\wedge g^{-1}\circ C\,,\qquad \nabla^g C\in\Gamma(\Sym^4(TM,\R))
\end{equation}
with $H=\pm 1$ or $H=0$.  In case $M$ has dimension $n=2$, the Blaschke metric $g$ defines a conformal structure and $M$ becomes a Riemann surface. The Pick form 
\[
C=Q+\bar{Q}\,,\qquad Q\in\Gamma(K^3)
\] 
being trace-free with respect to $g$, can be expressed in terms of a cubic differential $Q$. The second condition in \eqref{eq:GCspheres} now expresses the fact that $Q\in H^0(K^3)$ is holomorphic. The first condition in \eqref{eq:GCspheres}, when written with respect to a fixed conformal background metric $g_0$, that is  $g=e^{2u}g_0$ for a smooth function $u$ on $M$, gives the  Tzitz{\'e}ica or the Liouville equation 
\[
2\triangle_{g_0} u+2|Q|^2_{g_0} e^{-4u}-H e^{2u}-K^{g_0}=0
\]
depending whether $H=\pm 1$ or $H=0$. Here $K^{g_0}$ denotes the curvature function of the metric $g_0$. There are two natural choices for a background metric: first, a local flat metric 
$g_0=|dz|^2$ for a holomorphic chart $(U,z)$ which, with the holomorphic cubic differential $Q=qdz^3$, gives
\[
\triangle u +8|q|^2 e^{-4u}- \tfrac{1}{2}H e^{2u}=0\,,
\]
and second, the global hyperbolic metric $g_0$ of curvature $-1$ yielding
\begin{equation}\label{eq:Tzit-global}
2\triangle_{g_0} u +2|Q|^2_{g_0} e^{-4u}-H e^{2u}+1=0\,.
\end{equation}
Given a Riemann surface $M$ and a holomorphic cubic differential $Q\in H^0(K^3)$, the Tzitz{\'e}ica and Liouville equations
\eqref{eq:Tzit-global} are  {\em zero curvature} equations. They express the flatness of the connection
\[d_V=\begin{pmatrix} \nabla & H I\\ g & d_{\R}\end{pmatrix}
\]
on the rank $3$ bundle  $V=TM\oplus\underline{\R}$ over $M$.  Zero curvature equations play an important role in geometry and mathematical physics. In addition to solving such equations by methods of non-linear analysis, their special structure allows for solutions via integrable systems \cite{Hi-harm}, \cite{BFPP}, non-Abelian Hodge theory \cite{Hi}, \cite{S}, and loop group factorization techniques \cite{DPW},  \cite{HHT}. The latter methods aim to reduce these equations to a $\bar{\partial}$-problem, which significantly facilitates the understanding of the monodromy representation $\rho\colon \pi_1(M,p_0)\to {\bf SL}_3(\R)$ of the flat connection $d_V$. 

Since the integrability equation \eqref{eq:Tzit-global} for a parabolic affine sphere, $H=0$, is Liouville's equation, such a surface has a holomorphic Weierstrass representation similar to minimal surfaces in $\R^3$. We shall exclude this case from here on and focus entirely on the elliptic and hyperbolic affine spheres. The diagonal connection $\nabla^g\oplus d_{\R}$ on
$V=TM\oplus\underline{\R}$ then is metric with respect to the diagonal metric $g\oplus H dt^2$  and by \eqref{eq:LeviCivita}, \eqref{eq:Pick}  the flat connection $d_V$ has the decomposition 
\begin{equation}\label{eq:FT2spheres}
d_V=\begin{pmatrix} \nabla^g & 0 \\ 0 & d_{\R}\end{pmatrix}+\begin{pmatrix}0 & H I \\ g & 0\end{pmatrix}+ \begin{pmatrix} g^{-1}\circ C & 0\\ 0 & 0\end{pmatrix}\,.
\end{equation}
Since $I=Hg^{\dagger}$ with respect to $g\oplus H dt^2$ and $C$ is symmetric, the second and third terms are self-adjoint. 

Consider the complexified bundle $V\otimes\C=TM\otimes\C\oplus \underline{\C}$. Using the Riemann surface structure of $M$, we can further decompose $TM\otimes\C=K^{-1}\oplus \overline{K^{-1}}$ into the direct sum of the dual canonical bundle and its conjugate. Any conformal metric $g$ on $M$ is a Hermitian metric on the complex line bundle $K^{-1}$ identifying $\overline{K^{-1}}$ with $K$.  Rearranging the order of the line bundles, we obtain  the rank $3$ bundle 
\begin{equation}\label{eq:K-cyclic}
V\otimes\C=K^{-1}\oplus\underline{\C}\oplus K
\end{equation}
with direct sum (pseudo) Hermitian metric $h=g\oplus H dt^2\oplus g^{-1}$. The flat connection $d_V$ in \eqref{eq:FT2spheres} 
then has the form
\begin{equation}\label{eq:FT3spheres}
d_{V\otimes \C}=D+\Phi+\Phi^{\dagger}
\end{equation}
where  $D=\nabla^g\oplus d_{\C}\oplus (\nabla^g)^*$ is the diagonal (pseudo) Hermitian connection with respect to $h$ and the Higgs field $\Phi\in \Gamma(K{\bf sl}(V\otimes \C))$ is given by  
\begin{equation}\label{eq:Higgsfield}
\Phi=\begin{pmatrix} 0&{\bf 1}&0\\0&0&{\bf 1}\\Q&0&0\end{pmatrix}.
\end{equation}
Here we use the fact that the cubic form $C=Q+\bar{Q}$ with $Q\in\Gamma(K^3)$, whence $Q$ is a section of $K\Hom(K^{-1},K)$, and ${\bf 1}$ denotes the constant section 
in $\underline{\C}=K\Hom(\underline{\C},K^{-1})=K\Hom(K,\underline{\C})$. The flatness condition for the connection $d_V$ are Hitchin's \cite{Hi} self-duality equations 
\begin{equation}\label{eq:self-duality}
R^D+[\Phi\wedge \Phi^{\dagger}]=0\,,\qquad \dbar^D \Phi=0
\end{equation}
for the (pseudo) Hermitian connection $D$ and Higgs field $\Phi$, which reformulate the Tzitz{\'e}ica equation \eqref{eq:Tzit-global} and the holomorphicity of the Pick differential $Q\in H^0(K^3)$.  Note that the connection $D$ is the Chern connection of the metric $h$ due to the fact that it is metric and induces the holomorphic structure on the bundle $V\otimes \C$.
We summarize this discussion in the following lemma, versions of which can be found in \cite{DP}, \cite{LMcI}, and \cite{La}.
\begin{Lem}\label{lem:DP}
Let $M$ be a Riemann surface, $Q\in \Gamma(K^3)$ a cubic differential, $g$ a conformal  Riemannian metric on $M$,  and $V=K^{-1}\oplus \underline{\C}\oplus K$ the (pseudo) Hermitian rank 3 bundle with diagonal metric $h=g\oplus Hdt^2\oplus  g^{-1}$ where $H=\pm 1$. Then the following are equivalent:
\begin{enumerate}[(i)]
\item 
$g$ and $Q$ are the Blaschke metric and Pick differential for an affine sphere immersion $f\colon \tilde{M}\to \R^3$ equivariant with respect to a representation $\rho\colon \pi_1(M,p_0)\to {\bf SL}_3(\R)$;
\item
$g$ and $Q$ satisfy the Tzitz{\'e}ica equation \eqref{eq:Tzit-global};
\item
The diagonal connection $D=\nabla^{g}\oplus d_{\C}\oplus (\nabla^g)^*$ and the Higgs field $\Phi$ in \eqref{eq:Higgsfield} satisfy the self-duality equations \eqref{eq:self-duality}. 
\end{enumerate}
The representation $\rho$ is the monodromy
of the flat connection $D+\Phi+\Phi^{\dagger}$. Note that for elliptic affine spheres, $H=-1$, the structure group of the self-duality equations is the non-compact group 
${\bf SU}_{2,1}$, whereas for hyperbolic affine spheres, $H=1$, the structure group is ${\bf SU}_3$. 
\end{Lem}
Our goal is to construct solutions to the Tzitz{\'e}ica equation \eqref{eq:Tzit-global} on a thrice-punctured sphere $\Sigma=S^2\setminus \{p_1,p_2,p_3\}$ for holomorphic cubic differentials $Q\in H^0(K^3)$ with quadratic poles at the punctures. Due to the previous lemma, this amounts to solving the self-duality equations \eqref{eq:self-duality}. The basic tenet of our approach to solving those equations  is to use the correspondences of non-Abelian Hodge theory for which the Riemann surface and the structure group need to be compact. 
The latter is guaranteed once we restrict to hyperbolic affine spheres, $H=1$, which we shall assume from now on.  For the former, we need to extend the bundle 
\begin{equation}\label{eq:Vsum}
V\otimes\C=K^{-1}\oplus \C\oplus K
\end{equation} 
and the Higgs field $\Phi$  from $\Sigma$ to $S^2$.
Since the bundle $V\otimes \C$ over $\Sigma$ is holomorphically trivial ($\Sigma$ is non-compact), there is no unique extension. But there are guiding principles: 
\begin{itemize}
\item
A holomorphic bundle over $S^2$ splits as a sum of holomorphic line bundles. In view of \eqref{eq:Vsum} this suggests that  the extension of $V\otimes \C$ should be of the form $\mathcal{O}(-d)\oplus \mathcal{O}\oplus \mathcal{O}(d)$, see also Lemma \ref{lem:bundletype} below;  
\item 
The structure \eqref{eq:Higgsfield} of the Higgs field should persist and since $Q$ has quadratic poles at the punctures, that is $Q\in H^0(S^2, K^3 \mathcal{O}(2\frak{D}))$ for the pole divisor $\frak{D}=p_1+p_2+p_3$, we conclude $d=1$.  
\end{itemize}
These considerations give rise to the rank 3 bundle
\begin{equation}\label{eq:S2Higgsbundle}
W=\mathcal{O}(-1)\oplus \mathcal{O}\oplus \mathcal{O}(1)
\end{equation}
over $S^2$ and Higgs field $\Phi\in H^0(S^2, K{\bf sl}(W)\mathcal{O}(\frak{D}))$ of the form
\begin{equation}\label{eq:S2Higgsfield}
\Phi=\begin{pmatrix} 0 &{\bf 1} & 0\\ 0 & 0 &{\bf 1} \\ Q & 0 & 0 \end{pmatrix}.
\end{equation}
Here ${\bf 1}$ denotes the constant section of $K\mathcal{O}(-1)\mathcal{O}(\frak{D})=\mathcal{O}$, that is, a meromorphic $1$-form with values in $\mathcal{O}(-1)$ with simple poles, and thus non-trivial residues, at the punctures $p_k$.  Since $Q$ is a meromorphic cubic differential with quadratic poles at $p_k$, it can be viewed as a holomorphic section of $K^3 \mathcal{O}(2\frak{D})=\mathcal{O}$. By the natural inclusion $H^0(\mathcal{O})\subset H^0(\mathcal{O}(\frak{D}))$ and the isomorphism $\mathcal{O}(\frak{D})=K\mathcal{O}(2)\mathcal{O}(\frak{D})$, we see that $Q$ has the correct form to be an entry in the Higgs field $\Phi$. Because $Q\in H^0(\mathcal{O})=\C$, it  is holomorphic at $p_k$ as a 1-form and thus has no residues. Thus, the Higgs field $\Phi$ has maximal nilpotent residues 
\begin{equation}\label{eq:S2ResHiggsfield}
\Res_{p_k}\Phi= \begin{pmatrix} 0 &\Res_{p_k}{\bf 1}  & 0\\ 0 & 0 &\Res_{p_k}{\bf 1} \\ 0 & 0 & 0 \end{pmatrix}\in {\bf sl}(W_{p_k})\,.
\end{equation}
at the punctures $p_k\in S^2$. 

At this stage, we have started from a hyperbolic affine sphere over a thrice-punctured 2-sphere $\Sigma=S^2\setminus\{p_1,p_2,p_3\}$  and derived what is known as a {\em parabolic Higgs bundle} $(W,\Phi)$ over the 2-sphere $S^2$. Our construction of semi-flat Calabi--Yau metrics is based on solutions of the Tzitz{\'e}ica equation \eqref{eq:Tzit-global}, that is, self-duality solutions \eqref{eq:self-duality}, which form one of the three cornerstones of the non-Abelian Hodge correspondence. The other two cornerstones are Higgs bundles and surface group representations, and eventually we shall need the full triangle of correspondences. Since parabolic non-Abelian Hodge theory is more delicate and perhaps less well-known to an intended wider audience of this paper, we briefly review the basic background and refer the reader to \cite{S},  and \cite{KW}, \cite{B}  for further details.

Let $M$ be a compact Riemann surface and $\frak{D}=p_1+\dots +p_n$ the singularity divisor given by distinct points in $M$.
We shall denote the punctured Riemann surface by $\mathring{M}=M\setminus\supp\frak{D}$.
 A {\em parabolic vector bundle} is a holomorphic vector bundle $W\to M$ of rank $r$ together with a flag 
\[ 
W_p=F_p^{1} \supsetneq F_p^{2}\supsetneq \dots \supsetneq F_p^{m_p}\supsetneq \{0\}
\]
in the fiber over each singular point $p\in \supp{\frak{D}}$. Attached to each flag, there is 
an increasing sequence of real numbers, called the {\em parabolic weights},
\[
0\leq\alpha_p^1 < \alpha_p^2<\dots <\alpha_p^{m_p}<1\,.
\]
Note that a holomorphic vector bundle on a  punctured Riemann surface has no unique extension across the punctures. The extra data provided by the parabolic weight filtration generically characterize such an extension. 

A {\em parabolic Higgs bundle} is  a parabolic vector bundle $W\to M$ with a trace-free, meromorphic endomorphism-valued 1-form, the {\em Higgs field},
\[
\Phi\in H^0(M, K{\bf sl}(W)\mathcal{O}(\frak{D}))
\]
with at most simple poles along $\frak{D}$, such that the residues $\Res_{p_k}\Phi \in {\bf sl}(W_{p_k})$ preserve the flags $\{F_{p_k}^i\}_{i=1}^{m_{p_k}}$.  If the residues satisfy $\Res_{p_k}\Phi (F_{p_k}^{i})\subset F_{p_k}^{i+1}$, hence are {\em nilpotent}, the parabolic Higgs bundle is called {\em strongly parabolic}.  Two parabolic Higgs bundles $(W,\Phi)$ and $(\tilde{W},\tilde{\Phi})$ are {\em isomorphic}, if there exists a holomorphic bundle isomorphism $W\cong \tilde{W}$ intertwining the Higgs fields.

For the purposes of this paper, we shall only consider parabolic vector bundles $W$ with trivial flags 
$W_{p_k}\supset \{0_{p_k}\}$ and zero parabolic weights. In this case, we call the parabolic bundle {\em strongly parabolic} if  the residues $\Res_{p_k}\Phi \in {\bf sl}(W_{p_k})$ are nilpotent. Since all parabolic weights are zero,  the notion of  {\em parabolic degree} coincides with the usual degree of the bundle.  A parabolic Higgs bundle of $\deg W=0$  is then called {\em stable}, respectively {\em semi-stable}, if every $\Phi$-invariant holomorphic subbundle $E\subset W$ has $\deg E < 0$, respectively $\deg E \leq 0$. 

Note  that the 
parabolic Higgs bundle 
\begin{equation}\label{eq:Higgsbundle}
W=\mathcal{O}(-1)\oplus \mathcal{O}\oplus \mathcal{O}(1)\,,\qquad \Phi=\begin{pmatrix} 0 &{\bf 1} & 0\\ 0 & 0 &{\bf 1} \\ Q & 0 & 0 \end{pmatrix}
\end{equation}
derived in  \eqref{eq:S2Higgsbundle}, \eqref{eq:S2Higgsfield} is stable strongly parabolic, and all its residues \eqref{eq:S2ResHiggsfield} are maximally nilpotent, that is, have minimal polynomial $\chi(\lambda)=\lambda^3$.

In order to recover a solution to the Tzitz{\'e}ica equation \eqref{eq:Tzit-global} from the parabolic Higgs bundle \eqref{eq:Higgsbundle}, we first need to solve the self-duality equations
\begin{equation}\label{eq:self-duality-general}
F^{D}+[\Phi\wedge \Phi^{\dagger_h}]=0\,,\qquad \dbar^{D} \Phi=0
\end{equation}
for a Hermitian metric $h$ on $W_{| \mathring{M}}$. Here  $D=D^h$ denotes the Chern connection, that is, the unique Hermitian connection for $h$ on $\mathring{M}$ compatible with the holomorphic structure on $W$. The self-duality equations imply the flatness of the ${\bf SL}$-connection
\begin{equation}\label{eq:flat}
d_W=D+\Phi+\Phi^{\dagger_h}\,,
\end{equation} 
which in turn defines the surface group representation $\rho\colon \pi_1(M,p_0)\to {\bf SL}(W_{p_0})$ via its monodromy. 

The non-Abelian Hodge correspondences between Higgs bundles $(W,\Phi)$  and self-duality solutions \eqref{eq:self-duality-general}, and between the latter and flat connections, or equivalently, surface group representations $\rho\colon \pi_1(M,p_0)\to {\bf SL}_r(\C)$, are well established \cite{Hi}, \cite{S}. We collect  the relevant results only for the special case of strongly parabolic Higgs bundles of degree zero with trivial flags and weights at the punctures and refer the reader to \cite{S} for the more general setup. The centerpiece of these correspondences is the existence of a unique  harmonic Hermitian metric $h$ on the Higgs bundle or the flat bundle belonging to a representation of the fundamental group. 

In the parabolic setting, uniqueness of the harmonic metric $h$ requires 
a certain growth condition, called {\em tameness}, 
\begin{equation}\label{eq:tame-distance}
\sup_{p\in U^{\times}_{p_k}}\dist ( h_p,\tilde{h}_p)<C
\end{equation}
to be imposed at the punctures $p_k\in M$ for some $C>0$.  Here $z\colon U_{p_k}\to \mathbb{D}$ denotes a centered chart onto the open unit disk $\mathbb{D}\subset \C$ over which $W$ holomorphically trivializes. To calculate the distance, we view the Hermitian metrics---in the holomorphic trivialization---as maps into the symmetric space ${\bf SL}_r(\C)/{\bf SU}_r$.  

The comparison metric $\tilde{h}$, which we describe only for the parabolic Higgs bundle \eqref{eq:S2Higgsbundle}, is given as follows:  the generalized eigenspaces of $\Res_{p_k}\Phi$ define a full flag in $W_{p_k}$ and thus, by choosing some Hermitian inner product  in $W_{p_k}$,  a decomposition 
\[
W_{p_k}=\ell_{-2}\oplus\ell_0\oplus \ell_{2}
\]
into lines where $\ell_{-2}$ is the kernel of $\Res_{p_k}\Phi$. 

Extending those lines to local holomorphic line subbundles of $W$ over the chart domain $U_{p_k}$  provides the decomposition  
\[
W_{|U_{p_k}}=L_{-2}\oplus L_0\oplus L_2\,.
\]
Since $\mathcal{O}(-1)\subset W$ is contained in the kernel of $\Res \Phi$ globally, we can choose $L_{-2}=\mathcal{O}(-1)$ over $U_{p_k}$. The comparison metric in this splitting, with respect to a Hermitian metric making the splitting orthogonal,  is then given as 
\begin{equation}\label{eq:comp-metric} 
\tilde{h}=|\log r|^{-2}\oplus 1\oplus |\log r|^2
\end{equation}
with $r=|z|$ the distance in the chart $U_{p_k}$. Even though this description involves choices on which $\tilde{h}$ will depend, tameness \eqref{eq:tame-distance} is independent of these choices and thus well-defined. 

\begin{The}[Simpson, \cite{S}]\label{thm:nAH}
Let $M$ be a compact Riemann surface, $\frak{D}=p_1+\dots+  p_n$ the singularity divisor, $\mathring{M}=M\setminus \supp\frak{D}$, and $[\gamma_k] \in \pi_1(\mathring{M},p_0)$ generators of the loops around $p_k\in M$.  Then the following holds: 
\begin{enumerate}[(i)]
\item
A strongly stable parabolic Higgs bundle $(W,\Phi)$ on $M$ of degree zero with trivial flags and weights at the punctures admits a unique tame Hermitian metric $h$ on $W_{| \mathring{M}}$ with $\det_W h=1$, which induces a unique solution of the self-duality equations  \eqref{eq:self-duality-general} on $\mathring{M}$.  The resulting flat ${\bf SL}$-connection $d_W=D+\Phi+\Phi^{\dagger_h}$ on $W_{| \mathring{M}}$, with $D$ the Chern connection, is irreducible and gives rise to the irreducible monodromy representation $\rho\colon \pi_1(\mathring{M},p_0)\to {\bf SL}(W_{p_0})$ with unipotent monodromies $\rho(\gamma_k)$ around the punctures $p_k\in M$.   
\item
Every irreducible representation $\rho\colon \pi_1(\mathring{M},p_0)\to {\bf SL}_r(\C)$ with unipotent  $\rho(\gamma_k)$ arises 
from (i) for a unique degree zero strongly parabolic Higgs bundle $(W,\Phi)$ on $M$ with trivial flags and weights at the punctures $p_k\in M$.
\item
The minimal polynomials of the nilpotent residues $\Res_{p_k}\Phi\in {\bf sl}(W_{p_k})$ are the same as the minimal polynomials of the nilpotent parts $\rho(\gamma_k)-\mathrm{I}$ of the unipotent monodromy around the punctures $p_k\in M$. 
\end{enumerate}
\end{The}
\begin{Rem}\label{rem:functorial}
The above described correspondence descends to the appropriate quotients by the natural gauge actions on the objects---Higgs bundles, self-duality solutions, flat bundles, surface group representations---and is functorial.
\end{Rem}
\begin{Rem}\label{rem:S1-action}
The action of $S^1\subset \C$ on a stable strongly parabolic Higgs bundle given by $(W,\Phi)\mapsto (W,\mu\Phi)$ for unimodular $|\mu|=1$ does not change the tame metric $h$, but does change the resulting flat connection and representation.
\end{Rem}

In our setting of solutions to the Tzitz{\'e}ica equation \eqref{eq:Tzit-global}, Blaschke metrics, affine spheres, and Monge--Amp{\`e}re metrics,  the occurring monodromies take values in the real group ${\bf SL}_3(\R)$. Thus, we need to understand which parabolic Higgs bundles correspond to real representations.  This is well-documented for the non-parabolic case \cite{Hi2}, \cite{S-loc}, but less accessible in the parabolic setting. For the parabolic Higgs bundles relevant to this paper, we only need a weaker characterization for which we sketch a proof following \cite{S-loc}.
\begin{Pro}\label{pro:reality}
Let $M$ be a compact Riemann surface, $\frak{D}=p_1+\dots+  p_n$ the singularity divisor, $\mathring{M}=M\setminus \supp\frak{D}$, and $(W,\Phi)$ a stable strongly parabolic Higgs bundle of degree zero with trivial flag and weights. Then the corresponding representation $\rho\colon \pi_1(M,p_0)\to {\bf SL}_r(\C)$ is equivalent  to the complex conjugate representation $\bar{\rho}$,  if and only if there exists an isomorphism of parabolic Higgs bundles $(W,\Phi)\cong (W^*,\Phi^*)$.
\end{Pro}
\begin{proof}
Let $h$ be the unique tame metric for $(W,\Phi)$ over $\mathring{M}$ from Theorem~\ref{thm:nAH} giving rise to the flat connection $d_W=D+\Phi+\Phi^{\dagger_h}$ with monodromy representation  $\rho$. The complex conjugate Higgs bundle $\overline{(W,\Phi)}=(\bar{W},\Phi^{\dagger_h})$ over $\mathring{M}$, with holomorphic structure $D^{1,0}$, gives rise via $h$ to the same flat connection $d_W$. Hence, its monodromy is the complex conjugate representation $\bar{\rho}$. Since $D$ is Hermitian for $h$, we have 
\[
h((D^{1,0}+\Phi^{\dagger_h})\psi,\varphi)+h(\psi, (D^{0,1} \varphi-\Phi)\varphi)=\dbar h(\psi,\varphi)
\]
for sections $\psi,\varphi\in\Gamma(W_{|\mathring{M}})$. Since $D$ is the Chern connection, $D^{0,1}=\dbar_W$ is the holomorphic structure of $W$. Thus, we have the isomorphism of Higgs bundles
\begin{equation}\label{eq:h-iso}
h\colon \overline{(W,\Phi)}\to (W,-\Phi)^*
\end{equation}
over $\mathring{M}$. The dual $(W,-\Phi)^*=(W^*,\Phi^*)$ is a stable strongly parabolic Higgs bundle with tame metric $h^*$ and we have shown that it gives rise to the complex conjugate representation $\bar{\rho}$. The proposition now follows from Theorem~\ref{thm:nAH}, which assigns to a stable strongly parabolic Higgs bundle a unique representation and vice versa. 
\end{proof}
We are now in a position to apply the non-Abelian Hodge correspondence in Theorem \ref{thm:nAH} for stable strongly parabolic Higgs bundles to obtain solutions to the Tzitz{\'e}ica equation \eqref{eq:Tzit-global} for hyperbolic affine spheres modeled on the thrice-punctured sphere $\Sigma=S^2\setminus\{p_1,p_2,p_3\}$. 
\begin{Lem}\label{lem:3}
 Consider the stable strongly parabolic Higgs bundle $(W,\Phi)$ over $S^2$ given in \eqref{eq:Higgsbundle} with singularity divisor $\mathfrak{D}=p_1+p_2+p_3$ and meromorphic cubic differential $Q\in H^0(K^3\mathcal{O}(2\mathfrak{D}))$  having quadratic poles at the punctures $p_k$.  Then we have the following:
\begin{enumerate}[(i)]
\item
There exists a unique tame Hermitian metric $h$ of unit volume on $W_{|\Sigma}$ satisfying the self-duality equations \eqref{eq:self-duality-general}. Moreover, $h=h_{-1}\oplus h_0\oplus h_1$ is  diagonal with respect to $W=\mathcal{O}(-1)\oplus\mathcal{O}\oplus\mathcal{O}(1)$ with $h_{-1}=h_1^{-1}$ and $h_0=dt^2$. The irreducible representation $\rho\colon \pi_1(\Sigma,p_0)\to {\bf SL}_3(\R)$  corresponding to $(W,\Phi)$, described in Theorem~\ref{thm:nAH}, is a real representation. 
\item
The identification $K=K^{-1}_{S^2}=\mathcal{O}(2)=\mathcal{O}(-1)\mathcal{O}(\mathfrak{D})$, given by the (up to scale) unique section with simple zeros at $p_k$, defines a Riemannian metric $g$ on $\Sigma$ satisfying the Tzitz{\'e}ica equation \eqref{eq:Tzit-global} for hyperbolic affine spheres, H=1. The Riemannian metric $g$ is bounded with respect to the hyperbolic cusp metric $g_0$, which is obtained from the parabolic Higgs bundle  \eqref{eq:Higgsbundle} with $Q\equiv 0$.
\end{enumerate}
\end{Lem}
\begin{Rem}\label{rem:ray}
Since $H^0(K^3\mathcal{O}(2\mathfrak{D}))=\C$, the previous lemma constructs a $\C$-family of solutions to the Tzitz{\'e}ica equation \eqref{eq:Tzit-global} on the thrice-punctured sphere. Scaling the cubic differential $Q\mapsto \mu Q$ by a unimodular complex number $|\mu|=1$ does not change the solution $g$ of the Tzitz{\'e}ica equation. Moreover, the maximum principle applied  to the Tzitz{\'e}ica equation \eqref{eq:Tzit-global}  shows \cite[Proposition 4.1]{OT} that two solutions are isometric if and only if their cubic differentials relate via a unimodular scale. Therefore, we obtain  a ray $[0,\infty)$-worth of non-isometric Blaschke metrics $g$ on the thrice-punctured sphere. The asymptotics of the solutions at the punctures are the same as for the solutions constructed by Loftin \cite{L}. 
\end{Rem}
\begin{proof}
The existence and uniqueness of a tame metric $h$ with $\det_W h=1$ to the given stable strongly parabolic Higgs bundle $(W,\Phi)$ satisfying the self-duality equations \eqref{eq:self-duality-general} follows from Theorem~\ref{thm:nAH}. Applying a standard argument \cite{Hi}, \cite{Ba}, \cite{ColLi} for cyclic Higgs bundles, namely that the Higgs field $\Phi$ scales by $\zeta=e^{2\pi i/3}$  under  conjugation  by $c=\diag(\zeta, 1,\zeta^2)$ and therefore the unique tame metric $h$ and the Chern connection $D$  satisfy $c\cdot h=h$ and $c\cdot D=D$, shows that $h$ and $D$ are diagonal with respect to $W=\mathcal{O}(-1)\oplus\mathcal{O}\oplus\mathcal{O}(1)$. Furthermore, the Higgs field $\Phi$ is symmetric and $c$ is orthogonal with respect to the non-degenerate,  symmetric bilinear pairing $b\colon W\to W^*$, $b=b^*$, given by permuting $\mathcal{O}(-1)$ with $\mathcal{O}(1)$. 
This implies that $bh^{-1}=hb$ and,  since $\det W=\mathcal{O}$, the diagonal metric $h$ has the required form $h=h_{1}^{-1}\oplus dt^2 \oplus h_{1}$.  Since  $\Phi=\Phi^{t_b}$, the isomorphism $b$ is an isomorphism of Higgs bundles $(W,\Phi)\cong (W^*,\Phi^*)$ and Proposition~\ref{pro:reality}
implies that the representation $\rho$ corresponding to $(W,\Phi)$ is conjugate to $\bar{\rho}$. As we shall prove in Theorem~\ref{thm:charvar} (iii), this implies---in our special setting---that $\rho$ takes values in ${\bf SL}_3(\R)$.

To recover a solution of the Tzitz{\'e}ica equation \eqref{eq:Tzit-global} from the parabolic Higgs bundle $(W,\Phi)$, we use the isomorphism $K^{-1}=\mathcal{O}(-1)\mathcal{O}(\mathfrak{D})$. This  defines a bundle isomorphism $a\colon \mathcal{O}(-1)\to K^{-1}$ over $\Sigma$ given by multiplication with the unique (up to scale)  section with simple zeros at $\mathfrak{D}$. Gauging the Higgs bundle $(W,\Phi)$ by the diagonal gauge $a\oplus 1\oplus a^{-1}$ yields the Higgs bundle  $K^{-1}\oplus \underline{\C}\oplus K$ over $\Sigma$ with Higgs field given by \eqref{eq:Higgsfield}. We equip 
the bundle $K^{-1}\oplus \underline{\C}\oplus K$ over $\Sigma$ with the diagonal metric $g\oplus dt^2 \oplus g^{-1}$ obtained from $h$ by the  gauge $a\oplus 1\oplus a^{-1}$. Then the self-duality equations \eqref{eq:self-duality} are satisfied for the gauged Higgs bundle and by Lemma~\ref{lem:DP} the metric $g$ and cubic differential $Q$ solve the Tzitz{\'e}ica equation \eqref{eq:Tzit-global} on the punctured 2-sphere $\Sigma$. In other words, $g$ is the Blaschke metric and $Q$ the cubic form of an equivariant hyperbolic affine sphere. The asymptotics at the punctures of the metric $g$ is determined by the asymptotics of the tame Hermitian metric $h$ modified by the gauge $a\oplus 1\oplus a^{-1}$, that is $g=h_{1}^{-1}/|a|^2$. The same holds for the comparison metric $\tilde{h}$ on $\mathcal{O}(-1)$ which, after applying the gauge,  becomes the hyperbolic cusp metric $1/{r^2 |\log r|^2}$ near the punctures. For $Q\equiv 0$, the metric $g$ obtained from the above construction is the hyperbolic metric $g=g_0$ due to the fact that it solves the  Tzitz{\'e}ica equation \eqref{eq:Tzit-global}  for $u\equiv 0$. By tameness, $h$ is in bounded distance to the comparison metric $\tilde{h}$, whence the Blaschke metric $g$  is in bounded distance from the hyperbolic cusp metric $g_0$. We note that the simply connected hyperbolic affine sphere corresponding to $Q\equiv 0$ is one sheet of a  hyperboloid in $\R^3$.  The latter is a special case of the more general result by \cite{Berwald} that affine immersions with trivial Pick differential are quadrics. 
\end{proof}

At this juncture, we have constructed a $[0,\infty)$-family of mutually non-isometric solutions of the Tzitz{\'e}ica equation \eqref{eq:Tzit-global} on a thrice-punctured $2$-sphere $\Sigma=S^2\setminus\{p_1,p_2,p_3\}$. This family of solutions is parametrized by meromorphic cubic  differentials $Q\in H^0(S^2, K^3\mathcal{O}(2\mathfrak{D}))$, up to unimodular scalings, with quadratic poles along $\mathfrak{D}=p_1+p_2+p_3$. Lemma~\ref{lem:monodromy} then constructs a corresponding $\C$-family of flat special affine 3-manifolds 
equipped with Monge--Amp{\`e}re metrics, modeled on $B=\Sigma\times (0,1)$, a $3$-ball $S^2\times [0,1)$ from which the Y-vertex $\{p_k\}\times [0,1)$, $k=1,2,3$,  has been removed.
Since the monodromies $\rho\colon \pi_1(B)\to {\bf SL}_3(\R)$ of the flat special affine structures on $B$ are the monodromies of the non-congruent hyperbolic affine sphere immersions $f\colon \tilde{\Sigma}\to \R^3$ determined by the Tzitz{\'e}ica solutions \eqref{eq:Tzit-global} together with the cubic differential $Q$, we obtain a $\C$-family of mutually non-congruent flat special affine 3-manifold structures on $B$. 
In particular, the Monge--Amp{\`e}re metrics on $B$ are mutually non-isometric. Lemma~\ref{lem:3} describes the asymptotics of the  Tzitz{\'e}ica solutions at the punctures and thus the asymptotics of the family of Monge--Amp{\`e}re metrics at the Y-vertex. 

A flat special affine manifold $(B,\nabla,\det_B)$ with a Monge--Amp{\`e}re metric $g_B=\nabla d\phi$, for a strictly convex potential $\phi\colon B\to \R$, equips the tangent bundle $X=TB$ with the structure of a Calabi--Yau manifold  fibered over $B$ via the tangent bundle projection $\pi\colon X\to B$ with special Lagrangian flat fibers $\pi^{-1}(b)=T_b B$: the flat connection $\nabla$ on $B$ provides an integrable horizontal subbundle $H\subset T(TB)$ isomorphic to $\pi^*TB$  and thus $TX=\pi^*TB\oplus \pi^* TB$, with the second factor isomorphic to the fiber tangents $\ker d\pi\subset T(TB)$.  The almost complex structure 
\[
J\colon TX\to TX\,,\qquad J=\begin{pmatrix} 0&-1\\1&\;\;\,0\end{pmatrix}
\]
on $X$ is integrable since the horizontal subbundle  $H\subset T(TB)$ is integrable. Together with the diagonal metric $g=\pi^*g_B\oplus \pi^*g_B$
this makes $(X,g,J)$ into a K\"ahler manifold whose K\"ahler form vanishes on each of the summands $\pi^*TB$. In particular, the fibers of $\pi\colon X\to B$ are Lagrangian. The Monge--Amp{\`e}re condition ${\det}_B \nabla d\phi=1$ is equivalent \cite{Hi3} to Ricci-flatness of $g$ and thus $(X,g,J)$ is Calabi--Yau. The real part of the parallel holomorphic $n$-form $\Omega=\pi^*\det_B\oplus\, i J^*\pi^*\det_B$  vanishes along the fibers showing that the fibers are special Lagrangian. The Riemannian metric $g$ depends only on the base $B$ of the fibration $\pi\colon X\to B$ and hence the fibers are flat---justifying the terminology {\em semi-flat} for such fibrations.  

In order for the  metric $g$ on $X=TB$ to descend to a  Calabi--Yau metric on the total space of a special Lagrangian torus fibration $X/\Lambda\to B$, the full rank lattice bundle $\Lambda\subset TB$ has to be $\nabla$-parallel.  The existence of such a lattice bundle  is equivalent to the integrality of the monodromy representation $\rho\colon \pi_1(B)\to {\bf SL}_n(\Z)$ of the flat connection $\nabla$ via  $\Gamma=\tilde{B}\times_{\rho} \Z^n$. 

We summarize those observations in the following result, which is an elaboration---using non-Abelian Hodge theoretic methods---of a suggestion made in \cite{LYZerr}:
\begin{The}\label{lem:SL-fiber}
Let $\Sigma=S^2\setminus\{p_1,p_2,p_3\}$ be the thrice-punctured 2-sphere, $B=\Sigma\times (0,1)$ an open  3-ball from which the  Y-vertex $\{p_k\}\times [0,1)$ has been removed, and $X=TB$ the tangent bundle of $B$. 

Then there exists a $\C$-family, parametrized by meromorphic cubic differentials $Q\in H^0(K\mathcal{O}(2\mathfrak{D}))$ on $S^2$ with quadratic poles along $\mathfrak{D}=p_1+p_2+p_3$, of mutually non-isometric semi-flat 
Calabi--Yau metrics on $X$ whose fibers $\pi\colon X\to B$ are special Lagrangian 3-planes. 
\end{The}
The main contribution of this paper is to show that for certain choices of meromorphic cubic  differentials $Q$, these semi-flat 
Calabi--Yau metrics metrics descend to a special Lagrangian torus fibration $X/\Lambda\to B$ given by a parallel lattice subbundle $\Lambda\subset X$.
\begin{The}\label{thm:main}
Among the $\C$-family constructed in Theorem~\ref{lem:SL-fiber}, there exists an infinite series of non-isometric semi-flat 
Calabi--Yau metrics on $X=TB$ invariant under a parallel lattice bundle $\Lambda\subset TB$. In particular, there exists an infinite series of non-isometric  semi-flat
Calabi--Yau metrics on $\pi\colon X/\Lambda\to B$ fibered by special Lagrangian $3$-tori.
\end{The}
\begin{Rem}
The degeneration of the Calabi--Yau metric on $X/\Lambda$ along the Y-vertex is determined by the asymptotics of the solution to the Tzitz{\'e}ica equation \eqref{eq:Tzit-global}. It would be interesting to compare our analytically based approach to \cite[Remark 1.2]{YL} and also to algebro-geometric constructions \cite{Gross}, \cite{CJL}, where one considers simple pole singularities of the holomorphic volume form $\Omega$ on $X/\Lambda$ along divisors.
\end{Rem}

The remainder of the paper is concerned with the proof of Theorem~\ref{thm:main}. We again make use of the non-Abelian Hodge correspondence, albeit this time employing the relationship between parabolic Higgs bundles and surface group representations. Since our parabolic Higgs bundle is non-generic, we cannot directly apply known results about this correspondence.  
However, we shall develop some of the necessary properties---which are of interest on their own---specifically for our situation. 

The basic idea of the proof is to show that the map from Higgs bundles of the form \eqref{eq:Higgsbundle} into the Hitchin component of the ${\bf SL}_3(\R)$ representation variety of the thrice-punctured 2-sphere $\Sigma$ with 
unipotent monodromies is surjective---in fact, a homeomorphism. We  utilize an explicit algebraic description of this real variety to find integral representations.  Essential to this approach is Simpson's \cite{S}  bijection between the Betti moduli space of irreducible ${\bf SL}_r(\C)$ surface group representations of punctured Riemann surfaces with prescribed conjugacy classes at the punctures and the Dolbeault moduli space of rank $r$ stable parabolic Higgs bundles with the corresponding prescribed parabolic weight filtrations and residues of the Higgs fields at the punctures.

\section{The Betti moduli space}
We choose a base point $p_0\in\Sigma$ in the thrice-punctured sphere $\Sigma=S^2\setminus\{p_1,p_2,p_3\}$ and representative generators $\gamma_k$, $k=1,2,3$,  of the fundamental group $\pi_1(\Sigma, p_0)$ based at $p_0$.  It follows from \cite{S} that parabolic Higgs bundles of the form \eqref{eq:Higgsbundle}  give rise to representations $\rho\colon \pi_1(\Sigma, p_0)\to {\bf SL}_3(\C)$ for which the $\rho(\gamma_k)$ are  unipotent with minimal polynomial $\chi(\lambda)=(\lambda-1)^3$. This reflects the trivial parabolic weight filtrations of the bundles and the nilpotency of the residues of the Higgs fields \eqref{eq:S2ResHiggsfield}. 
\begin{Def}\label{def:Betti}
The Betti moduli space in our set up is given by ${\bf SL}_3(\C)$ conjugacy classes of representations $\rho\colon  \pi_1(\Sigma, p_0)\to {\bf SL}_3(\C)$, whose generators $\rho(\gamma_k)$, $k=1,2,3$, are unipotent:
\[
\mathcal{M}_{B}=\{\rho\colon  \pi_1(\Sigma,p_0)\to {\bf SL}_3(\C)\,;\, \rho(\gamma_k)\,\text{ unipotent}\}/{\bf SL}_3(\C).
\]
We denote by $\mathcal{M}_{B}^{s}\subset \mathcal{M}_{B}^{ps}\subset \mathcal{M}_{B}$ the subspaces of conjugacy classes of irreducible and completely reducible, that is, sums of irreducible, representations.
\end{Def}
The fundamental group relation of the thrice-punctured sphere implies
\begin{equation}\label{eq:FGrelation}
\rho(\gamma_3)=(\rho(\gamma_2)\rho(\gamma_1))^{-1}
\end{equation}
for the generators of $[\rho]\in \mathcal{M}_{B}$. Therefore, we have the description of the Betti moduli space 
\begin{equation}\label{eq:pairs}
\mathcal{M}_{B}=\{(A_1,A_2)\,;\, A_1, A_2, (A_2 A_1)^{-1}\,\text{unipotent}\}/{\bf SL}_3(\C)
\end{equation}
as  pairs $(A_1,A_2)$ of ${\bf SL}_3(\C)$ matrices subject to the unipotency condition. This description equips $\mathcal{M}_{B}$ with the quotient topology, which is Hausdorff on the subspace $\mathcal{M}_{B}^{ps}$ of completely reducible representations. The subspace of irreducible representations $\mathcal{M}_{B}^{s}$ is a smooth complex manifold via this quotient structure.  

An important aspect of our approach is an explicit description of  the completely reducible part $\mathcal{M}_{B}^{ps}$ of the Betti moduli space as an affine surface $\mathcal{F}\subset \C^3$. Consider the following character map $\mathcal{X}\colon \mathcal{M}_{B}\to \C^3$ defined in Lawton \cite{Law}, \cite{Law07} given by
\begin{equation}\label{eq:characters}
\mathcal{X}_{\rho}=
\begin{pmatrix}
\tr (\rho(\gamma_1)\rho(\gamma_2)^{-1})\\
\tr (\rho(\gamma_1)^{-1}\rho(\gamma_2))\\
\tr (\rho(\gamma_1)\rho(\gamma_2)\rho(\gamma_1)^{-1}\rho(\gamma_2)^{-1})
\end{pmatrix}
\end{equation}
and let  $P\in\Z[x,y,z]$ be Lawton's polynomial, the cubic polynomial with integer coefficients
\begin{equation}\label{eq:polynom}
P=414-108 x+x^3-108 y+21 x y+y^3-(51-9x-9 y+x y)z+z^2\,.
\end{equation}
We then have the following result characterizing the Betti moduli space:
\begin{The}\label{thm:charvar}
\begin{enumerate}[(i)]
\item
The cubic affine surface $\mathcal{F}=P^{-1}(0)\subset \C^3$ is smooth away from the point $(3,3,3)\in \mathcal{F}$ and has the birational representation
\[
\Psi\colon\C^2\to \mathcal{F}\,,\quad \Psi(s,t)=\left(3+\tfrac{(3+3s+t)^2}{st-s^3-1},3+s\tfrac{(3+3 s+t)^2}{st-s^3-1},3+t\tfrac{(3+3 s+t)^2}{st-s^3-1}\right)\,.
\]
\item
The map $\mathcal{X}\colon \mathcal{M}_{B}\to \C^3$ surjects onto $\mathcal{F}$ and its restriction $\mathcal{X}\colon \mathcal{M}_{B}^{ps}\to \mathcal{F}$ to completely reducible representations  is a homeomorphism. Moreover, every representation in $\mathcal{M}_{B}^{ps}$, except the trivial representation $\rho=\mathrm{I}$, is irreducible and $\mathcal{X}_{\mathrm{I}}=(3,3,3)$. In particular, the character map restricts to a biholomorphism $\mathcal{X}\colon \mathcal{M}_{B}^{s}\to \mathcal{F}\setminus\{(3,3,3)\}$.
\item
The $2$-dimensional variety of real points $\mathcal{F}(\R)=\mathcal{F}\cap \R^3$ of the cubic affine surface $\mathcal{F}\subset \C^3$ has two connected components. Their preimages $\mathcal C_1,\mathcal C_2\subset \mathcal{M}_{B}^{ps}$ under $\mathcal{X}\colon \mathcal{M}_{B}^{ps}\to \mathcal{F}$ are the  completely reducible ${\bf SL}_3(\R)$ representations 
$\mathcal{M}_{B}^{ps} (\R)=\mathcal C_1\cup \mathcal C_2$ of the thrice-punctured sphere subject to the 
 unipotency constraint. The component $\mathcal C_1$ contains the trivial representation and $\mathcal C_2$ is the ${\bf SL}_3(\R)$ Hitchin component which contains the uniformization representation.  In particular, $\mathcal C_2$ is smooth. 
\end{enumerate}
\end{The}
Lawton's polynomial \eqref{eq:polynom}, defining the character variety $\mathcal{F}$, and bijectivity  of the character map $\mathcal{X}\colon \mathcal{M}^{ps}_B\to \mathcal{F}$, are deduced from 
the more general setting in \cite{Law}, \cite{Law07} using our trace conditions  \eqref{eq:tracecon} on the generators $\rho(\gamma_k)$. Since we need more detailed knowledge about the character map---the set of smooth points, the number of real connected components---not given explicitly in  \cite{Law}, \cite{Law07}, we will provide proofs for our specific setting.
\begin{proof}
Part (i) follows by direct inspection. For part (ii), we construct an explicit inverse for the 
character map $\mathcal{X}\colon \mathcal{M}_{B}^{ps}\to \mathcal{F}$, which is also needed to describe the ${\bf SL}_3(\R)$---and later ${\bf SL}_3(\Z)$---representations in $\mathcal{M}_{B}^{ps}$. Part (iii) will follow from analyzing the connected components of the restriction to $\R^2$ of the birational parametrization $\Psi$.  
\begin{Fac}\label{fac:1}
We recall the following basic facts about ${\bf SL}_3(\C)$ matrices:
\begin{enumerate}[(i)]
\item
$A\in {\bf SL}_3(\C)$ has characteristic polynomial  $\chi_A(\lambda)=(\lambda-1)^3$ if and only if the trace conditions $\tr A=\tr A^2=3$ hold;
\item
If $A\in {\bf SL}_3(\C)$ has $\tr A=3$, then $\tr A^{-1}=3$ if and only if $\tr A^2=3$.
\end{enumerate}
\end{Fac}
The first item is a consequence of Cayley-Hamilton
\[
\det(\lambda-A)=\lambda^3-\tr A\, \lambda^2+\tfrac{1}{2}((\tr A)^2-\tr A^2)\lambda+\det A,
\]
while the second item follows from 
\[
9=(\tr A)^2=\tr A^2+2\tr A^{-1}
\]
using $\det A =1$.

A representation $[\rho]\in\mathcal{M}_{B}$ is described by the pair  $(\rho(\gamma_1),\rho(\gamma_2))$ of ${\bf SL}_3(\C)$ matrices subject to the  unipotency conditions \eqref{eq:pairs} which, using Observation~\ref{fac:1}, can be rephrased for $k=1, 2$ by the more manageable trace conditions
\begin{equation}\label{eq:tracecon}
\tr\rho(\gamma_k)=\tr \rho(\gamma_k)^2=3\,,\qquad \tr \rho(\gamma_2)\rho(\gamma_1)= \tr (\rho(\gamma_2)\rho(\gamma_1))^2=3\,.
\end{equation}

After a conjugation, we may assume that $\rho(\gamma_1)$ is equal to one of three possible Jordan forms, depending on whether the geometric multiplicity of the eigenvalue $\lambda=1$ is one, two, or three:
\begin{equation}\label{eq:Jordan}
A_1=\begin{pmatrix} 1&1&0\\0&1&1\\0&0&1\end{pmatrix}\,,\quad \tilde{A}_1=\begin{pmatrix} 1&0&0\\0&1&1\\0&0&1\end{pmatrix}\,,\quad \text{or}\quad \mathrm{I}\,.
\end{equation}
If $\rho(\gamma_1)=\mathrm{I}$ and $\rho$ is completely reducible, then $\rho$ is the trivial representation.  We hence assume $\rho(\gamma_1)\neq \mathrm{I}$ and determine the possible conjugators which commute with $A_1$, respectively $\tilde{A}_1$, in order to obtain a suitable normal form for $\rho(\gamma_2)$. 
\begin{Fac}\label{fac:conjugator} 
We have the following properties which follow by solving a linear system of equations for $g$ respectively $\tilde{g}$:
\begin{enumerate}[(i)]
\item 
$g^{-1}A_1g=A_1$ for $g\in {\bf SL}_3(\C)$, if and only if
$
g=\begin{pmatrix} g_{11}& g_{12}& g_{13}\\0& g_{11} & g_{12}\\ 0&0& g_{11}\end{pmatrix}.
$
\item
$\tilde{g}^{-1}\tilde{A}_1\tilde{g}=\tilde{A}_1$ for $\tilde{g}\in {\bf SL}_3(\C)$, if and only if 
$
\tilde{g}=\begin{pmatrix} \tilde{g}_{22}^{-2}& 0& \tilde{g}_{13}\\\tilde{g}_{21}& \tilde{g}_{22}& \tilde{g}_{12}\\ 0&0&\tilde{g}_{22}\end{pmatrix}.
$
\end{enumerate}
\end{Fac}
Next, we use the conjugators fixing $\rho(\gamma_1)$, determined by Observation~\ref{fac:conjugator}, to obtain a normal form for $\rho(\gamma_2)$ by direct computation.
\begin{Fac}\label{fac:normalform}
Let $A_1$, $\tilde{A}_1$ be given in \eqref{eq:Jordan}, let
$A_2=(a_{ij})$, $\tilde{A}_2=(\tilde{a}_{ij})$ be arbitrary $3\times 3$ matrices, and let $g,\tilde{g}\in{\bf SL}_3(\C)$ be as in Observation~\ref{fac:conjugator}.
\begin{enumerate}[(i)]
\item If $a_{31}\neq 0$, then there exists a unique (up to scale) $g$  such that
\begin{equation}\label{eq:nf11}
g^{-1} A_ 2 g=\begin{pmatrix}0& b_{12}& b_{13}\\ 0& b_{22}& b_{23}\\ b_{31}& b_{32}& b_{33} \end{pmatrix}.
\end{equation}
\item
If $a_{31}=0$, then there either exists a $1$-dimensional common invariant subspace of $A_ 1$ and $A_ 2$, or there is
a unique (up to scale) $g$ such that 
\begin{equation}\label{eq:nf12}
g^{-1}{A_ 2}g=\begin{pmatrix}0& b_{12}& b_{13}\\ b_{21}& b_{22}& 0\\ 0& b_{32}& b_{33} \end{pmatrix}.
\end{equation}
\item
If $\tilde{a}_{32}\neq0$, then there exists a  unique (up to scale) $\tilde{g}$  such that
\begin{equation}\label{eq:nf21}
\tilde{g}^{-1}\tilde{A}_2\tilde{g}=\begin{pmatrix}\tilde{b}_{11}& 0& \tilde{b}_{13}\\ \tilde{b}_{21}& 0& \tilde{b}_{23}\\ 0& \tilde{b}_{32}& \tilde{b}_{33} \end{pmatrix}.
\end{equation}
\item
If $\tilde{a}_{32}=0$, then there either exists a common invariant subspace of $\tilde{A}_1$ and $\tilde{A}_2$ of dimension $1$ or $2$, or there is a unique (up to  scale) $\tilde{g}$  such that
\begin{equation}\label{eq:nf22}
\tilde{g}^{-1}\tilde{A}_2\tilde{g}=\begin{pmatrix}0& \tilde{b}_{12}& \tilde{b}_{13}\\ 0& 0& \tilde{b}_{23}\\ 1& 0& \tilde{b}_{33} \end{pmatrix}.
\end{equation}
\end{enumerate}
\end{Fac}
Putting those observations together, we can characterize the irreducible and reducible representations in $\mathcal{M}_B$ via the character map 
$\mathcal{X}\colon \mathcal{M}_{B}\to \C^3$ defined by \eqref{eq:characters}.
\begin{Fac}\label{fac:characterization}
Let  $[\rho]\in \mathcal{M}_B$. Then  $\rho$ is irreducible if and only if $\mathcal{X}_{\rho}\neq (3,3,3)$.  In particular, $\rho$ is reducible if and only if $\mathcal{X}_{\rho}=(3,3,3)$. In this case, the trivial representation $\rho=\mathrm{I}$  is the only completely reducible presentation.
\end{Fac} 
Note that if $\rho$ is reducible and satisfies the unipotency conditions \eqref{eq:tracecon}, then  a direct computation shows $\mathcal{X}_{\rho}=(3,3,3)$.
The converse statement,
we shall carry out in detail only for $\rho(\gamma_1)=A_1$ in \eqref{eq:Jordan}. The reasoning for $\rho(\gamma_1)=\tilde{A}_1$ is simpler and follows along the same lines.  By Observation~\ref{fac:normalform}, there are two cases to consider, depending on whether $a_{31}\neq 0$ or $a_{31}=0$. In the first case, $\rho(\gamma_2)$ can be conjugated into the normal form $B_2$ as in \eqref{eq:nf11} and there is no common invariant subspace. The trace conditions \eqref{eq:tracecon} unravel to 
\[
\tr\rho(\gamma_1)\rho(\gamma_2)=\tr A_1B_2=b_{22}+b_{33}+b_{32}=3=\tr\rho(\gamma_2)=\tr B_2=b_{22}+b_{33}
\]
which yields $b_{32}=0$. But then 
\[
\tr \rho(\gamma_1)^{-1}\rho(\gamma_2)=\tr A_1^{-1}B_2=b_{22}+b_{33}+b_{31}=\tr\rho(\gamma_2)+b_{31}=3+b_{31}\neq 3
\]
since $b_{31} \neq 0$ and due to  \eqref{eq:characters}  the  character map $\mathcal{X}_{\rho}  \neq (3,3,3)$.   

We now look at the second possibility, that is, $\rho(\gamma_1)=A_1$ has $a_{31}=0$. If $a_{21}\neq 0$, then there is no common invariant 1-dimensional subspace (which would have to be $\C e_1$) and we can work with the normal form \eqref{eq:nf12}. Using again the trace conditions \eqref{eq:tracecon}, we have
\[
\tr\rho(\gamma_1)\rho(\gamma_2)=\tr A_1B_2=b_{21}+b_{32}+b_{22}+b_{33}=3=\tr\rho(\gamma_2)=\tr B_2=b_{22}+b_{33}
\]
and we can conclude
\[
b_{32}=-b_{21}\,.
\]
Since $b_{32}=-b_{21}\neq 0$, there cannot be a 2-dimensional common invariant subspace for $A_1$ and $B_2$. Moreover, one can directly calculate
\begin{equation}\label{eq:characters-nf12}
\mathcal{X}_{\rho}=(3-b_{21}^2, 3, 3+3 b_{21}^2-b_{21}^3)\,,
\end{equation}
which shows that the character map $\mathcal{X}_{\rho}\neq (3,3,3)$. Finally, if $a_{21}=0$, then $\C e_1$ is a 1-dimensional common invariant subspace. The trace condition $\tr\rho(\gamma_1)\rho(\gamma_2)=3$ then implies $a_{32}=0$ and one calculates that
\[
\mathcal{X}_{\rho}=(3, 3, 3)\,.
\]
We should point out that there are reducible, but not completely reducible, representations $\rho\in\mathcal{M}_{B}$ which we know have to satisfy $\mathcal{X}_{\rho}=(3,3,3)$. But as we already have remarked, the trivial representations $\rho=\mathrm{I}$ is the only completely reducible such representation.

We are now in a position to prove part (ii) of Theorem~\ref{thm:charvar}. 
First, we show that the character map $\mathcal{X}$ is well-defined and surjects onto $\mathcal{F}$. Let $\mathcal{X}_{\rho}=(x,y,z)\in\C^3$ and assume $(x,y,z)\neq (3,3,3)$ since we already know that $\mathcal{X}_{\mathrm{I}}=(3,3,3)$.

{\bf Case 1}: $y\neq 3$. The proof of Observation~\ref{fac:characterization} shows that after a conjugation, $\rho(\gamma_1)=A_1$ has the first Jordan normal form in \eqref{eq:Jordan} and  $\rho(\gamma_2)$ is given by \eqref{eq:nf11} with $b_{32}=0$. To determine the remaining entries $b_{ij}$ of 
$\rho(\gamma_2)$, we repeatedly use the trace conditions \eqref{eq:tracecon}. From \eqref{eq:characters} we have $y=\tr\rho(\gamma_1)^{-1}\rho(\gamma_2)$ which, together with $\tr \rho(\gamma_2)=3$, implies
\[
b_{33}=3-b_{22}\quad\text{and}\quad b_{31}=y-3.
\]
Evaluating $\tr \rho(\gamma_2)^2=3$ gives
\[
b_{13}=\frac{3-3b_{22}+b_{22}^2}{3-y}\,.
\]
With these, $\tr(\rho(\gamma_1)\rho(\gamma_2))^2=3$ is equivalent to
\[ 
(b_{12}+b_{22}+b_{23})(y-3)=0
\]
and therefore
\[
b_{23}=-b_{12}-b_{22}.
\]
Scaling the first component of the character map $\eqref{eq:characters}$ by $\det\rho(\gamma_2)$, we solve
\[
\tilde{x}=x\det\rho(\gamma_2)=\det\rho(\gamma_2)\tr \rho(\gamma_1)\rho(\gamma_2)^{-1}
\]
for $b_{22}$ to obtain
\[
b_{22}=-\frac{\tilde{x}-3}{y-3}\,.
\]
We now make the substitution 
\begin{equation}\label{eq:substitute}
 b_{12}=\frac{3\tilde{x}+3y+\tilde{z}-21}{(y-3)^2}
 \end{equation}
 for a yet to be determined $\tilde{z}$. Then 
\[
1-\det\rho(\gamma_2)=\frac{P(\tilde{x},y,\tilde{z})}{(y-3)^3}
\]
with Lawton's polynomial $P$ given in \eqref{eq:polynom}. Therefore, $\det\rho(\gamma_2)=1$ if and only if $(\tilde{x},y,\tilde{z})\in \mathcal{F}$ and $\tilde{x}=x$.  Using \eqref{eq:substitute}, the third component 
\[
z=\tr \rho(\gamma_1)\rho(\gamma_2)\rho(\gamma_1)^{-1}\rho(\gamma_2)^{-1}
\] of the character map calculates to 
\[
\tilde{z}-\det\rho(\gamma_2)z=3\frac{P(\tilde{x},y,\tilde{z})}{(y-3)^3}
\]
which shows $z=\tilde{z}$ since $\det\rho(\gamma_2)=1$. Thus, we have shown that the character map has image in $\mathcal{F}$. Conversely, given $(x,y,z)\in \mathcal{F}$ with $y\neq 3$, the above discussion constructs an irreducible representation 
$[\rho]\in\mathcal{M}_B$ with generators $\rho(\gamma_1)=A_1$ and
\begin{equation}\label{eq:A2-nf11}
\rho(\gamma_2)=\begin{pmatrix}
0&\frac{3 x+3 y+z-21}{(y-3)^2}&-\frac{63+x^2-15 x-27 y+3 xy+3 y^2}{(y-3)^3}\\
0&-\frac{x-3}{y-3}&\frac{30-6x -6 y+xy-z }{(y-3)^2}\\
y-3&0&3+\frac{x-3}{y-3}
\end{pmatrix}
\end{equation}
satisfying $\mathcal{X}_{\rho}=(x,y,z)$. This representation is unique, since $P(x,y,\tilde{z})=0$ is quadratic in $\tilde{z}$, and only one of the two solutions satisfies $z=\tr \rho(\gamma_1)\rho(\gamma_2)\rho(\gamma_1)^{-1}\rho(\gamma_2)^{-1}$.

{\bf Case 2}: $y=3$ and $x\neq 3$. Like in the previous case, the proof of Observation~\ref{fac:characterization}, in particular \eqref{eq:characters-nf12}, shows that after a conjugation $\rho(\gamma_1)=A_1$ has the first Jordan normal form in \eqref{eq:Jordan}, but now $\rho(\gamma_2)$ is given by \eqref{eq:nf12} with $b_{21}\neq 0$.  We calculate the remaining entries $b_{ij}$  of $\rho(\gamma_2)$: the trace conditions \eqref{eq:tracecon}
\[
\tr\rho(\gamma_2)=\tr\rho(\gamma_1)\rho(\gamma_2)=\tr\rho(\gamma_2)^2=3
\]
imply that
\[
b_{33}=3-b_{22}\,,\quad b_{32}=-b_{21}\quad \text{and}\quad b_{12}=\frac{-3+3b_{22}-b_{22}^2}{b_{21}}\,.
\]
The remaining trace condition $\tr(\rho(\gamma_2)\rho(\gamma_1))^2=3$ gives us
\[
b_{22}=3-b_{21}\,,
\]
and $\det \rho(\gamma_2)=1$ unravels to 
\[
b_{13}=\frac{(b_{21})^3-1}{b_{21}^2}\,.
\]
Thus, given $b_{21}\neq 0$, we  have determined the unique irreducible representation $\rho(\gamma_1)=A_1$ and 
\begin{equation}\label{eq:A2-nf12}
\rho(\gamma_2) =\begin{pmatrix} 0& 3-\tfrac{3}{b_{21}}-b_{21}&\tfrac{(b_{21}-1)^3}{b_{21}^2}\\
b_{21}& 3-b_{21}&0\\
0&-b_{21}&b_{21}
\end{pmatrix}.
\end{equation}
The character map evaluates to $\mathcal{X}_{\rho}=(3-b_{21}^2, 3, 3+3b_{21}^2-b_{21}^3)$, which surjects to  $\mathcal{F}\cap \{y=3\}$
 for $b_{21}\in\C$.

{\bf Case 3}: $y=3$ and $x=3$. Then $P(x,y,z)=0$ implies $z=3$ and we can take $\rho=\mathrm{I}$.

We thus have shown that the character map $\mathcal{X}\colon \mathcal{M}_{B}^{ps}\to \mathcal{F}$ is bijective with $\mathcal{X}_{\mathrm{I}}=(3,3,3)$. The restriction to the smooth irreducible locus $\mathcal{X}\colon \mathcal{M}_{B}^{s}\to \mathcal{F}\setminus\{(3,3,3)\}$ is a holomorphic bijection onto the smooth locus of $\mathcal{F}$. Applying invariance of domain, the inverse map is continuous. In fact, since we have computed explicit inverses of $\mathcal{X}\colon \mathcal{M}_{B}^{s}\to \mathcal{F}\setminus\{(3,3,3)\}$ along and away from the divisor $y=0$, one could show with a little more work that the inverse map is also holomorphic. We contend ourselves that $\mathcal{X}\colon \mathcal{M}_{B}^{s}\to \mathcal{F}\setminus\{(3,3,3)\}$ is a homeomorphism and, extending  $\mathcal{X}$ into the trivial representation, we get  the homeomorphism $\mathcal{X}\colon \mathcal{M}_{B}^{ps}\to \mathcal{F}$, thereby finishing the proof of part (ii) of Theorem~\ref{thm:charvar}.

At last we come to part (iii) of the theorem. From \eqref{eq:A2-nf11} and \eqref{eq:A2-nf12}, we deduce that a completely reducible representation $\rho\colon\pi_1(\Sigma,p_0)\to {\bf SL}_3(\C)$ is conjugated to an ${\bf SL}_3(\R)$ representation if and only if  $\mathcal{X}_{\rho}$ takes values in the real points $\mathcal{F}(\R)=\mathcal{F}\cap \R^3$ of the character variety. In other words, the homeomorphism $\mathcal{X}\colon \mathcal{M}_{B}^{ps}\to \mathcal{F}$ restricts to a homeomorphism
\begin{equation}\label{eq:realrep}
\mathcal{X}\colon\mathcal{M}_B^{ps}(\R)\to \mathcal{F}(\R)\,.
\end{equation}
Moreover, the birational parametrization $\Psi$  given in Theorem~\ref{thm:charvar} restricts to the real birational map $\Psi\colon \R^2\setminus\{st-s^3-1=0\}\to \mathcal{F}(\R)$.  We then have the disjoint union
\begin{equation}\label{eq:impsi}
\begin{gathered}
\mathcal{F}(\R)=\im \Psi\, \cup \,D_1\,\cup \, D_2\\
D_1=\{(x,y,z)\in \mathcal{F}(\R)\,;\, x=3\}\\
 D_2=\{(x,y,z)\in \mathcal{F}(\R)\,;\,\ y=6-x,\, z=3\}\,.
\end{gathered}
\end{equation}
The complement of the singular set $\{st-s^3-1\neq 0\}\subset \R^2$ has three connected components
\[
S_1=\{s<0,\, st-s^3-1> 0\}\,,\quad \tilde{S}_1=\{st-s^3-1<0\}\,,\quad S_2=\{s>0,\, st-s^3-1>0\}\,,
\]
whose images under $\Psi$ in $\mathcal{F}(\R)\setminus D$, where  $D=D_1\cup D_2$,  are
\[
\Psi(S_1)=\{x>3,y<3\}\,,\quad
\Psi(\tilde S_1)=\{x<3\}\,,\quad
\Psi(S_2)=\{x>3,y>3\}.
\]
Now $D_2=\{(3-u,3+u,3)\,;\, u\in\R\}\subset \mathcal{F}(\R)$ is a line and for $(3,y,z)\in D_1$ we can check that $y<3$. Therefore,
\[
C_1=\Psi(S_1)\cup \Psi(\tilde{S_1})\cup D \subset \mathcal{F}(\R)
\]
is connected.  Furthermore, for $(s,t)\in S_2$ we have 
\begin{equation*}
x=3+\frac{(3+3s+t)^2}{s t-s^3-1}\geq 3+\frac{6 st}{s t}\geq 9\,,\quad 
y=3+s\frac{(3+3s+t)^2}{s t-s^3-1}\geq 3+\frac{6t}{t}\geq9\,,
\end{equation*}
whereas $x\geq 3$ implies $y\leq 3$ for $(x,y,z)\in C_1$.
This shows that the closures of $C_1$ and $C_2=\Psi(S_2)$ are disjoint.

Thus, the real character variety $\mathcal{F}(\R)=C_1\cup C_2$ has two connected components. We already know that the real character map \eqref{eq:realrep} is a homeomorphism and thus 
\[
\mathcal{M}_B^{ps}(\R)=\mathcal{C}_1\cup\mathcal{C}_2
\]
has two connected components $\mathcal{C}_k=\mathcal{X}^{-1}(C_k)$. From \eqref{eq:impsi}, we see that 
the trivial representation $[\mathrm{I}]\in \mathcal{M}_B^{ps}(\R)$ lies in $\mathcal{C}_1$. 

It remains to verify that $\mathcal{C}_2$ is the ${\bf SL}_3(\R)$ Hitchin component in  $\mathcal{M}_B^{ps}(\R)$. This is the connected component of the representation $\rho_0\colon \pi_1(\Sigma,p_0)\to {\bf SL}_3(\R)$ obtained from the monodromy representation $\rho_u\colon\pi_1(\Sigma,p_0) \to{\bf SL}_2(\R)$  of the developing map of the thrice-punctured sphere $\Sigma$ into the upper half-plane $\mathbb{H}^2$ via the $2:1$ homomorphism ${\bf SL}_2(\R)\to {\bf SL}_3(\R)$. This homomorphism is given by the irreducible representation of ${\bf SL}_2(\R)$ on the symmetric product $\R^2\odot\R^2$. In the geometric context of hyperbolic affine spheres and solutions to the Tzitz{\'e}ica equation \eqref{eq:Tzit-global}, the representation $\rho_0$ is the monodromy of the hyperbolic affine sphere $f\colon \tilde{\Sigma}\to\R^3$ corresponding to trivial Pick differential $Q\equiv0$ as seen in Lemmas~\ref{lem:DP} and \ref{lem:3}. The image of $f$ is one sheet of a 2-sheeted hyperboloid and composing $f$ with the hyperbolic stereographic projection to the  Poincar{\'e} disk, followed by the standard isometry with $\mathbb{H}^2$, yields the developing map. The 
monodromy $\rho_u$ of the latter, up to conjugation, is  
the principal congruence subgroup of level $2$ generated by 
\[
\rho_u(\gamma_1)=\begin{pmatrix} 1&2\\0&1\end{pmatrix},\qquad 
\rho_u(\gamma_2)=\begin{pmatrix} 1&0\\-2&1\end{pmatrix}.
\]
The corresponding ${\bf SL}_3(\R)$ representation $\rho_0$ then calculates to 
\begin{equation*}\label{eq:unirep}
\rho_0(\gamma_1)=\begin{pmatrix} 1&2&0\\0&1&16\\0&0&1\end{pmatrix},\quad
\rho_0(\gamma_2)=\begin{pmatrix} 0&1&-7\\0&-1&8\\1&0&4\end{pmatrix},
\quad\rho_0(\gamma_3)=\begin{pmatrix} 12&14&1\\-8&-9&0\\1&1&0\end{pmatrix},
\end{equation*}
which is an integral representation. The character map evaluates to  $\mathcal{X}_{\rho_0}=(35,35,323)$ and hence $[\rho_0]\in \mathcal{C}_2$.  
\end{proof}
\begin{Rem}\label{rem:maxuni}
From the above discussion, we have seen that the conjugacy classes of the generators $\rho(\gamma_k)$, $k=1,2,3$, of irreducible representations $[\rho]\in\mathcal{M}_B^{s}$ are all given by $A_1$ in \eqref{eq:Jordan}, that is, are maximally unipotent.
\end{Rem}
We now address the issue of integral representations $\rho\colon \pi_1(\Sigma,p_0)\to {\bf SL}_3(\Z)$,  which is the main ingredient in proving the existence of Calabi--Yau metrics on special Lagrangian torus fibrations with singularity over a Y-vertex in Theorem~\ref{thm:main}. A necessary condition for $[\rho]\in\mathcal{M}_B^{ps}(\R)$ to be integral is that the character map $\mathcal{X}$  takes values in integer points $\mathcal{F}(\Z)$ of the character variety, a Diophantine problem. 
\begin{The}\label{thm:Diophantine}
There are infinitely many integral representations $[\rho]\in\mathcal{M}_B^{ps}(\Z)$ in the connected component  of the trivial representation $\mathcal{C}_1$ and in the Hitchin component $\mathcal{C}_2$ of the real Betti moduli space $\mathcal{M}_B^{ps}(\R)$.
\end{The}
\begin{proof}
The birational parametrization 
\[
\Psi\colon\R^2\to \mathcal{F}(\R)\,,\quad \Psi(s,t)=\left(3+\tfrac{(3+3s+t)^2}{st-s^3-1},3+s\tfrac{(3+3 s+t)^2}{st-s^3-1},3+t\tfrac{(3+3 s+t)^2}{st-s^3-1}\right)
\]
from Theorem~\ref{thm:charvar} allows us to calculate explicit integer points in $\mathcal{F}(\R)$, which will give rise to integral representations.

For $n\in\Z$, the points 
\[
\Psi(n,n^2)=(3-(3+3n+n^2)^2,3-n(3+3n+n^2)^2,3-n^2(3+3n+n^2)^2)
\]
are contained in $\mathcal{F}(\Z)$, in fact, as can be seen from \eqref{eq:impsi}, are all contained in the component $C_1$. Using the explicit expression \eqref{eq:A2-nf11} of the inverse of the character map we obtain, after a conjugation, the corresponding, non-congruent, integral representations
\[
\rho(\gamma_1)=\begin{pmatrix}1&3+3n+n^2&0\\0&1&3+3n+n^2\\0&0&1\end{pmatrix},\qquad 
\rho(\gamma_2)=\begin{pmatrix}0&0&-1\\1&0&3+n\\-n&-1&3\end{pmatrix}
\]
in the component $\mathcal{C}_1$. 

To obtain integral representations in the Hitchin component $\mathcal{C}_2$ turns out to be more subtle. For $k,l\in \N$, the ansatz 
\[
(s,t)=\left(\frac{1}{k},\frac{l}{k}\right)\,,\qquad k>0\,,\quad l>k^2+\frac{1}{k}
\]
has $st-s^3-1>0$ and thus
\[
\Psi(s,t)=\left(3+\tfrac{k (3 k+l+3)^2}{k l-k^3 -1},3+\tfrac{(3 k+l+3)^2}{k l-k^3 -1},3+\tfrac{l (3 k+l+3)^2}{ k l-k^3 -1}\right)
\]
takes values in $C_2$ by \eqref{eq:impsi}. Moreover, 
$\Psi(s,t)\in \mathcal{F}(\Z)$  if and only if there exist $m\in\N$ with
\begin{equation}\label{eq:m}
(3 k+l+3)^2=m(kl-k^3-1)\,.
\end{equation}
That this equation has infinitely many natural number solutions $k,l,m\in\N$ with $l>k^2+1/k$, we learned from Alekseyev \cite{Al}. For $n\in\N$ define recursively 
\begin{equation}\label{eq:rec}
u_{n+1}:=23 u_n-u_{n-1}-4\,,\qquad u_0=1\,,\quad u_1=2 
\end{equation}
and let
\[
k_n=u_n\,,\quad 
l_n=u_n^2+u_{n-1}\,,\quad 
m_n=(u_{n+1}+2)(u_n+2)+24\,.
\]
Then the sequences $k_n,l_n>0$ are  increasing and $l_n>k_n^2+1/k_n$ for all $n\in\N$. By induction using \eqref{eq:rec}, we have
\[
(3+3k_n+l_n)^2-m_n(k_n l_n-k_n^3-1)=(1 + u_{n}^2) b_{n-1}
\]
with
\[
b_{n}=29 + u_{n}^2 + u_{n} (4 - 23 u_{n-1}) + u_{n-1} (4 + u_{n-1}).
\]
Applying \eqref{eq:rec} once more, we obtain $b_{n+1}=b_n$ for all $n\in\N$, which implies $b_n=0$ due to $b_1=0$. 

A natural number solution $k,l,m\in\N$ of \eqref{eq:m} with $l>k^2+1/k$ yields via \eqref{eq:A2-nf11} the representation
\[
\rho(\gamma_1)=A_1\,,\qquad 
\rho(\gamma_2)=\begin{pmatrix} 0&\tfrac{3 k+l+3}{m} &-\tfrac{k^2+3 k+3}{m}\\ 0&-k&-\tfrac{3 k+l+3 -km}{m}\\m&0&3+k\end{pmatrix}
\]
with $A_1$ as in \eqref{eq:Jordan}.  Writing  $m=ab^2$ with $a$ square free---thus $ab$ divides $3k+l+3$---and conjugating this representation by 
\[
C=\begin{pmatrix}\tfrac{b}{m}&0& \tfrac{bk}{m}\\0&1&0\\0&0&b\end{pmatrix}
\]
gives the ${\bf SL}_3(\Z)$ representation
\[
C^{-1}\rho(\gamma_1)C=\begin{pmatrix}1&\tfrac{m}{b}& 0\\0&1&b\\0&0&1\end{pmatrix},\qquad 
C^{-1}\rho(\gamma_2)C=\begin{pmatrix}-k&\tfrac{3k+l+3}{b}& -3(1+k)^2\\0&-k& -\tfrac{b(3k+l+3-km)}{m}\\1&0&2k+3\end{pmatrix}.
\]
\end{proof}
\begin{Rem}
There are many more integral representations $[\rho]\in\mathcal{C}_2$ in the Hitchin component corresponding to integer points in the character variety $\mathcal{F}$ which are not contained in the series constructed in Theorem~\ref{thm:Diophantine}. We list some of them in the table below:

\[
\begin{array}{|c|c|c|c|c|c|c|}
  \hline
  (s,t) &  & x & y & z   &  \rho(\gamma_1) & \rho(\gamma_2) \\ 
  \hline
  \hline
 (1,3) &  &84 & 84 & 256 & \begin{pmatrix}
1&9&0\\
0&1&9\\
0&0&1
\end{pmatrix} &  \begin{pmatrix}
0&1&-7\\
0&-1&8\\
1&0&4
\end{pmatrix}\\
\hline
(3,20)& &35&99&643&\begin{pmatrix}
1&11&32\\
0&97&288\\
0&-32&-95
\end{pmatrix}&\begin{pmatrix}
0&1&3\\
0&21&64\\
1&-6&-18
\end{pmatrix}\\
  \hline
 (\tfrac{7}{5},\tfrac{18}{5})& &93&129&327& 
  \begin{pmatrix}
1&18&-5\\
0&1&1\\
0&0&1
\end{pmatrix}&\begin{pmatrix}
0&1&-1\\
2&-1&2\\
7&-1&-4
\end{pmatrix}\\
\hline
    \end{array}
    \]

At this point, we do not know whether all integer points $\mathcal{F}(\Z)\subset \mathcal{F}$ in the character variety give rise to integral representations, nor can we characterize all integer points.
\end{Rem}

\section{The Dolbeault moduli space}
In order to complete the proof of Theorem~\ref{thm:main}, we will construct a 
homeomorphism from the $\C$-family $(\mathcal{O}(-1)\oplus\mathcal{O}\oplus\mathcal{O}(1),\Phi^Q)$   of stable strongly parabolic Higgs bundles \eqref{eq:Higgsbundle} over $S^2$ with singularity divisor $\mathfrak{D}=p_1+p_2+p_2$ parametrized by meromorphic cubic differentials $Q$ with quadratic poles at the punctures $p_1,p_2,p_3\in S^2$, 
to the ${\bf SL}_3(\R)$ Hitchin component $\mathcal{C}_2\subset \mathcal{M}_B$.  Theorem~\ref{thm:Diophantine} then provides an infinite series of such parabolic Higgs bundles \eqref{eq:Higgsbundle},  whose associated flat connections \eqref{eq:flat} have integral monodromy representations $\rho\colon \pi_1(\Sigma,p_0)\to{\bf SL}_3(\Z)$. Lemma~\ref{lem:3} then guarantees an infinite series of mutually non-isometric solutions of the Tzitz{\'e}ica equation \eqref{eq:Tzit-global} on the thrice-punctured sphere $\Sigma$ with integral monodromy.  But this is precisely the condition needed for the Calabi--Yau metric on $X=TB$, constructed from the Monge--Amp{\`e}re metric in Lemma~\ref{lem:monodromy} on the three ball $B=\Sigma\times (0,1)$ with the Y-vertex $\{p_k\}\times [0,1)$ deleted, to descend to a special Lagrangian torus fibration $X/\Lambda\to B$ for a parallel lattice subbundle $\Lambda\subset TB$.

\begin{Def}\label{def:Dol}
Let $\Sigma=S^2\setminus\{p_1,p_2,p_3\}$ be the thrice-punctured sphere. The Dolbeault moduli space 
$\mathcal{M}_D$
in our setting is given by isomorphism classes of  rank $3$ strongly parabolic Higgs bundles $(W,\Phi)$ over $S^2$ of degree zero with trivial flags and weights at the punctures $p_k$. Recall that strongly parabolic means that the residues $\Res_{p_k}\Phi\in {\bf sl}(W_{p_k})$ are nilpotent. 

In accordance with our Definition~\ref{def:Betti}, and justifying the notation there, we denote  by $\mathcal{M}_D^s\subset  \mathcal{M}_D^{ps}\subset \mathcal{M}_D$ the stable and polystable loci in the Dolbeault space.
\end{Def}

From the non-Abelian Hodge correspondence for parabolic bundles  \cite{S}, it follows that there is a bijection $\mathcal{M}_B^{ps}\cong \mathcal{M}_{D}^{ps}$ 
between the Betti space of completely reducible representations given in Definition~\ref{def:Betti} and the Dolbeault space of polystable strongly parabolic bundles. This bijection maps the trivial representation to the trivial parabolic Higgs bundle $(\mathcal{O}^{\oplus 3},0)$ and restricts to a bijection
\[
\mathcal{M}_B^{s}\cong \mathcal{M}_{D}^{s}\colon [\rho]\leftrightarrow [(W,\Phi)]
\]
between the Betti space of  irreducible representations and the Dolbeault space of stable strongly parabolic Higgs bundles. Via this bijection, as stated in Theorem~\ref{thm:nAH}, the unipotency orders of the generators $\rho(\gamma_k)$ around the punctures of $[\rho]\in \mathcal{M}_B^{s}$ correspond to the nilpotency orders of the residues  $\Res_{p_k}\Phi\in {\bf sl}(W_{p_k})$ of the Higgs field of the corresponding $[(W,\Phi)]\in \mathcal{M}_{D}^{s}$.  Therefore, our study of the Betti moduli  space in the previous chapter, in particular Remark~\ref{rem:maxuni}, shows that for $[(W,\Phi)]\in \mathcal{M}_{D}^{s}$ the residues $\Res_{p_k}\Phi$ of the Higgs field all are maximally nilpotent.

Our approach is to first explicitly characterize all stable strongly parabolic Higgs bundles $(W,\Phi)$  over the 2-sphere with singularity divisor $\mathfrak{D}=p_1+p_2+p_3$ up to isomorphisms.  Among those we then identify the Higgs bundles which give rise to real representations. 
\begin{Lem}\label{lem:bundletype}
Let $(W,\Phi)$ be a rank $3$ stable parabolic Higgs bundle of degree zero over the $2$-sphere $S^2$ with trivial flags and weights at the three punctures $p_k\in S^2$. Then $W=\mathcal{O}^{\oplus 3}$ or 
$W=\mathcal{O}(-1)\oplus\mathcal{O}\oplus\mathcal{O}(1)$.
\end{Lem}
\begin{proof}
Since every holomorphic bundle over $S^2$ is a direct sum of holomorphic line bundles, we must have
\begin{equation}\label{eq:Gsplit}
W=\mathcal O(-k-l)\oplus\mathcal O(l)\oplus\mathcal O(k)\,,\qquad k\geq 0\,,\qquad  k\geq l\geq -k-l\in\Z\,.
\end{equation}
Since $K_{S^2}=\mathcal{O}(-2)$,  the order of the entries $\Phi_{i j}$ of the Higgs field $\Phi$ with respect to the splitting above are given by
\begin{equation}\label{eq:orderentries}
\ord \Phi=\begin{pmatrix}-2&-2l-k-2&-l-2k-2\\2l+k-2&-2&l-k-2\\2k+l-2&k-l-2&-2 \end{pmatrix}.
\end{equation}
Furthermore, the Higgs field $\Phi\in H^0(K{\bf sl}(W)\mathcal{O}(\frak{D}))$ has at most simple poles along the singularity divisor $\mathfrak{D}=p_1+p_2+p_3$, and thus any entry $\Phi_{ij}$ with respect to the splitting \eqref{eq:Gsplit} has to vanish if $\ord \Phi_{ij}<-3$. Assuming $k>l+1$, we obtain
\[
-l-2k-2\leq l-k-2<-3
\]
which shows that $\mathcal{O}(k)$ is a $\Phi$-invariant subbundle of $W$ of non-negative degree, contradicting stability.
Next, we assume $2l+k>1$, which implies
\[
-l-2k-2<-2l-k-2<-3
\] 
and then $\mathcal O(l)\oplus \mathcal O(k)$ is a $\Phi$-invariant subbundle of non-negative degree $k+l\geq 0$, again contradicting stability.
This leaves only the two bundles $W$ listed in the lemma.
\end{proof}
Knowing the possible bundle types of stable parabolic Higgs bundles $(W,\Phi)$ over the 2-sphere $S^2$, we now determine the possible Higgs fields which, as we already know, must have maximal nilpotent residues at the punctures $p_k\in S^2$. 
\begin{Lem}\label{lem:cyclicHiggs}
Let $[(W,\Phi)]\in\mathcal{M}_B^s$ with underlying bundle $W=\mathcal{O}(-1)\oplus\mathcal{O}\oplus\mathcal{O}(1)$. Then there is a unique meromorphic cubic differential $Q\in H^0(K^3\mathcal{O}(2\frak{D}))$ with quadratic poles along the singularity divisor  $\frak{D}=p_1+p_2+p_3$,  so that 
\[
\Phi=\begin{pmatrix} 0 &{\bf 1} & 0\\ 0 & 0 &{\bf 1} \\ Q & 0 & 0 \end{pmatrix}\in H^0(K{\bf sl}(W)\mathcal{O}(\frak{D}))
\]
is the Higgs field described in \eqref{eq:S2Higgsfield} and \eqref{eq:S2ResHiggsfield}. 
\end{Lem}
\begin{proof}
We write the Higgs field $\Phi=(\Phi_{ij})$ as a matrix in the decomposition $W=\mathcal{O}(-1)\oplus\mathcal{O}\oplus\mathcal{O}(1)$. Using  \eqref{eq:orderentries}, we see that $\Phi_{13}=0$. If $\Phi_{23}=0$ or $\Phi_{12}=0$, then $\mathcal{O}(1)$ or $\mathcal{O}\oplus\mathcal{O}(1)$ are $\Phi$-invariant subbundles of positive degree, contradicting stability. Since $\Phi_{12},\Phi_{23}\in H^0(K\mathcal{O}(-1)\mathcal{O}(\frak{D}))=H^0(\mathcal{O})$, we conclude $\Phi_{12}=\Phi_{23}={\bf 1}$ after a scaling. 
Next we claim that there is a unique  holomorphic bundle isomorphism $a\in H^0({\bf SL}(W))$, given in  the decomposition $W=\mathcal{O}(-1)\oplus\mathcal{O}\oplus\mathcal{O}(1)$ by 
\[
a=\begin{pmatrix} 1&0&0\\a_{21}&1&0\\ a_{31}&a_{32}&1\end{pmatrix}\,,
\]
such that $a^{-1}\Phi a$ has the form \eqref{eq:S2Higgsfield} as stated in the lemma. This amounts to solving linear equations successively for the entries of $a$. For instance, the first equation we get is 
\[
\Phi_{11}+a_{21}{\bf 1}=0\,,
\]
which has the solution $a_{21}=-\Phi_{11}$. Note that this is well defined, since $a_{21},\Phi_{11}\in H^0(\mathcal{O}(1))$. Continuing this process,  we can determine all entries of $a$ uniquely and 
\[
\Phi=\begin{pmatrix} 0 &{\bf 1} & 0\\ 0 & 0 &{\bf 1} \\ \Phi_{31}& \Phi_{32} & 0 \end{pmatrix}\,.
\]
So far, we have not used the maximal nilpotency of the residues $\Res_{p_k}\Phi$, which in our situation is characterized by 
\[
\tr (\Res_{p_k}\Phi)^2=\tr (\Res_{p_k}\Phi)^3=0\,.
\]
Since $\Res_{p_k}{\bf 1}\neq 0$, these conditions unravel to $\Res_{p_k} \Phi_{31}=\Res_{p_k} \Phi_{32}=0$. Now $\Phi_{31}\in H^0(K\mathcal{O}(2)\mathcal{O}(\frak{D}))=H^0(\mathcal{O}(\frak{D}))$ is a function with at most simple poles at $p_k$ and therefore $\Phi_{31}\in\C=H^0(K^3\mathcal{O}(2\frak{D}))$ has to be constant. Similarly, $\Phi_{32}\in H^0(\mathcal{O}(-1)\mathcal{O}(\frak{D}))$ is without residues at $p_k$ and thus, as a holomorphic section of $\mathcal{O}(-1)$, has to vanish.
\end{proof}
For the classification of stable strongly parabolic Higgs bundles $(\mathcal{O}^{\oplus 3},\Phi)$, let $z$ denote the unique affine chart  on $S^2$ with $z(p_1)=0$, $z(p_2)=1$ and $z(p_3)=\infty$.
\begin{Lem}\label{lem:trivialHiggs}
Consider $[(W,\Phi)]\in\mathcal{M}_B^s$ with underlying bundle $W=\mathcal{O}^{\oplus 3}$ and denote by $\ell_k=\ker\Res_{p_k}\Phi \subset W_{p_k}$ the 1-dimensional kernel of the maximally nilpotent residues of the Higgs field at the punctures $p_k\in S^2$.
Then, up to constant ${\bf SL}_3(\C)$ conjugation, which is unique up to a scaling by a cube root of unity, we have the following families of  mutually non-isomorphic Higgs fields characterized by the configurations of the kernels $\ell_k$ of the residues:
\begin{enumerate}[(i)]
\item
If $W=\ell_1\oplus\ell_2\oplus \ell_3$ and $\ker(\Res_{p_1}\Phi)^2\cap \ell_2\oplus \ell_3\neq \ell_2, \ell_3$, then
\begin{equation}\label{eq:Higgsgen}
\Phi=
x \begin{pmatrix} 
0& y+1&y\\
0& 1& 1\\
0& -1& -1\\
\end{pmatrix}\frac{dz}{z}+
x\begin{pmatrix} 
-1& 0& -y\\
-\tfrac{1}{1+y}& 0& -1\\
\tfrac{1}{y}& 0& 1\\
\end{pmatrix}\frac{dz}{z-1} 
\end{equation}
for unique parameters $(x,y)\in\C^2\setminus\{(0,0),(0,-1)\}$.
\item
If $W=\ell_1\oplus\ell_2\oplus \ell_3$ and $\ker(\Res_{p_1}\Phi)^2\cap \ell_2\oplus \ell_3=\ell_2$, then
\begin{equation}\label{eq:Higgstriv+}
\Phi=
\begin{pmatrix} 0& 1&0\\0&0&1\\
0&0&0\end{pmatrix}\frac{dz}{z}+
\begin{pmatrix} 
0& 0&0 \\0&0&-1\\
x&0&0\end{pmatrix}\frac{dz}{z-1}
\end{equation}
for a unique parameter $x\in\C^{\times}$.
\item
If $W=\ell_1\oplus\ell_2\oplus \ell_3$ and $\ker(\Res_{p_1}\Phi)^2\cap \ell_2\oplus \ell_3=\ell_3$, then
\begin{equation}\label{eq:Higgstriv-}
\Phi=
\begin{pmatrix} 0& 0&1\\0&0&0\\
0&1&0\end{pmatrix}\frac{dz}{z}+
\begin{pmatrix} 
0& 0&-1 \\x&0&0\\
0&0&0\end{pmatrix}\frac{dz}{z-1}
\end{equation}
for a unique parameter $x\in\C^{\times}$.
\item
If $\dim \ell_1\oplus\ell_2\oplus \ell_3\leq 2$, then 
\begin{equation}\label{eq:Higgstriv0}
\Phi=
\begin{pmatrix} 0& 0&x\\0&0&0\\
0&1&0\end{pmatrix}\frac{dz}{z}+
\begin{pmatrix} 
0& 0&0 \\0&0&-x\\
1&0&0\end{pmatrix}\frac{dz}{z-1}\end{equation}
for a unique parameter $x\in\C^{\times}$.
 \end{enumerate}
\end{Lem}
\begin{proof}
First assume that the kernels of the residues $W=\ell_1\oplus\ell_2\oplus \ell_3$ span $W$. We can conjugate so that  $\ell_k=\C e_k$ for the standard basis $e_k\in\C^3$, which fixes any further conjugations to be diagonal. Then the residues at $p_1,p_2$ are of the  form
\[
\Res_{p_1}\Phi=\begin{pmatrix} 0&a_1&b_1\\0&a_2&b_2\\0&a_3&-a_2\end{pmatrix}\,,
\qquad 
\Res_{p_2}\Phi=\begin{pmatrix} -a_2&0&-b_1\\c_2&0&-b_2\\c_3&0&a_2\end{pmatrix}\,,
\]
where we used that $\Res_{p_3}\Phi=-\Res_{p_1}\Phi-\Res_{p_2}\Phi$. Next we evaluate the nilpotency conditions on the three residues 
\begin{equation}\label{eq:restracecon}
\begin{gathered}
\tr(\Res_{p_1}\Phi)^2=2(a_2^2+a_3b_2)=0\,,\qquad \tr(\Res_{p_2}\Phi)^2=2(a_2^2-b_1c_3)=0\\
\tr(\Res_{p_3}\Phi)^2=2\tr(\Res_{p_1}\Phi\Res_{p_2}\Phi)=2(a_2^2+a_1c_2)=0
\end{gathered}
\end{equation}
and, using the first equation above,  calculate
\begin{equation}\label{eq:res1^2}
(\Res_{p_1}\Phi)^2=\begin{pmatrix} 0& a_1a_2+a_3b_1& a_1b_2-a_2b_1\\0&0&0\\0&0&0\end{pmatrix}.
\end{equation}
We now assume the generic configuration  $\ker(\Res_{p_1}\Phi)^2\cap \ell_2\oplus \ell_3\neq \ell_2, \ell_3$. Then, given $\alpha\neq 0$, there is a unique diagonal conjugation so that 
\[
a_1a_2+a_3b_1=a_1b_2-a_2b_1=\alpha\,,
\]
which geometrically  means $\ker(\Res_{p_1}\Phi)^2\cap \ell_2\oplus \ell_3=\C(e_2-e_3)$. Combining these last two equations with the trace conditions \eqref{eq:restracecon} gives $a_2=-a_3=b_2$ and thus 
\[
\Res_{p_1}\Phi=\begin{pmatrix} 0&a_1&b_1\\0&a_2&a_2\\0&-a_2&-a_2\end{pmatrix}.
\]
 Since $a_1b_2-a_2b_1=\alpha\neq 0$ and $a_2=b_2$, we have $a_2\neq 0$ and $a_2(a_1-b_1)=\alpha$.  Therefore, choosing $\alpha=a_2^2$ and setting $x:=a_2$, $y:=b_1/a_2$, we obtain 
 \[
 \Res_{p_1}\Phi=x \begin{pmatrix} 
0& y+1&y\\0& 1& 1\\0& -1& -1\\
\end{pmatrix} 
\]
Note that $y\neq -1$, since otherwise $a_1=0$ which in turn would imply $a_2=0$ by the third equation in \eqref{eq:restracecon}. Finally, 
using \eqref{eq:restracecon} again,  we determine the remaining coefficients in $\Res_{p_2}\Phi$ in terms of the parameters $(x,y)\in\C^2\setminus\{(0,0),(0,-1)\}$. 

The simpler configurations, where $\ker(\Res_{p_1}\Phi)^2\cap \ell_2\oplus \ell_3$ are either $\ell_2$ or $\ell_3$, are analyzed in the same way. In both cases, stability of the Higgs bundle $(W,\Phi)$  implies that the parameter $x\neq 0$. 
 
 At last we consider  $\dim \ell_1\oplus\ell_2\oplus \ell_3\leq 2$. If two of the lines $\ell_k$ are the same, then all three have to coincide, and $\mathcal{O}\subset W$ is a $\Phi$-invariant subbundle of zero degree, contradicting stability of $(W,\Phi)$.  Hence, we may assume $\ell_3\subset \ell_1\oplus \ell_2$ and $\ell_3\neq \ell_1,\ell_2$.  We conjugate $\Phi$ so that $\ell_1=\C e_1$, $\ell_2=\C e_2$, $\ell_3=\C(e_1-e_2)$, and choose a yet to be determined line $\C v$ transverse to $\ell_1\oplus\ell_2$. Then the residues of the Higgs field $\Phi$ are given by 
 \[
\Res_{p_1}\Phi=\begin{pmatrix} 0&a_1&b_1\\0&a_2&b_2\\0&a_3&-a_2\end{pmatrix}\,,
\qquad 
\Res_{p_2}\Phi=\begin{pmatrix} a_1&0&d_1\\a_2&0&d_2\\a_3&0&-a_1\end{pmatrix}\,,
\]
and
\[
\Res_{p_3}\Phi=-\Res_{p_1}\Phi-\Res_{p_2}\Phi=
\begin{pmatrix} -a_1&-a_1&-d_1\\-a_2&-a_2&-d_2\\-a_3&-a_3&a_1+a_2
\end{pmatrix}.
\]
As before, we evaluate the nilpotency conditions on the residues
\begin{equation}\label{eq:tracecon0}
\begin{gathered}
\tr(\Res_{p_1}\Phi)^2=2(a_2^2+a_3b_2)=0\,,\qquad \tr(\Res_{p_2}\Phi)^2=2(a_1^2+a_3d_1)=0\\
\tr(\Res_{p_3}\Phi)^2=2\tr(\Res_{p_1}\Phi\Res_{p_2}\Phi)=2(2a_1a_2 + a_3(b_1+d_2))=0\,.
\end{gathered}
\end{equation}
Subtracting the third equation from the first two, we obtain $a_1=a_2$. If $\ker(\Res_{p_1}\Phi)^2=\ell_1\oplus \ell_2$, we conclude from \eqref{eq:res1^2} that $a_1^2+a_3b_1=0$. Maximal nilpotency implies  $(\Res_{p_1}\Phi)^2\neq 0$ and, using the first trace condition \eqref{eq:tracecon0}, we deduce that $a_3=0$ and hence $a_1=a_2=0$, a contradiction to $(\Res_{p_1}\Phi)^2\neq 0$. We therefore may assume that $\ker(\Res_{p_1}\Phi)^2$ contains a line $\C v$ transverse to $\ell_1\oplus\ell_2$. We choose $v=\Res_{p_1}\Phi\,e_2$, which gives us $a_1=a_2=0$, $a_3=1$, $b_2=d_1=0$,  $d_2=-b_1$ and thus
\[
\Res_{p_1}\Phi=\begin{pmatrix} 0&0&b_1\\0&0&0\\0&1&0\end{pmatrix}\,,\qquad 
\Res_{p_2}\Phi=\begin{pmatrix} 0&0&0\\0&0&-b_1\\1&0&0\end{pmatrix}.
\]
Setting $x=b_1$ and noticing that $b_1\neq 0$ due to maximal nilpotency, we obtain the  family \eqref{eq:Higgstriv0}.
\end{proof}
Recall that our objective is to show that the $\C$-family of stable strongly parabolic Higgs bundles $(\mathcal{O}(-1)\oplus \mathcal{O}\oplus \mathcal{O}(1),\Phi^{Q})$ in Lemma~\ref{lem:cyclicHiggs}, parametrized by meromorphic cubic differentials $Q\in H^0(K^3\mathcal{O}(2\frak{D}))$ with quadratic poles along  $\frak{D}=p_1+p_2+p_3$, maps surjectively onto the Hitchin component $\mathcal{C}_2\subset \mathcal{M}^s_B(\R)$. For that, we need to understand the image $\mathcal{M}^s_D(\R)$ of the locus of real representations $\mathcal{M}^s_B(\R)$ under the bijection 
\[
\mathcal{M}_B^{s}\cong \mathcal{M}_{D}^{s}\colon [\rho]\leftrightarrow [(W,\Phi)]\,.
\]
\begin{Lem}\label{lem:realDol}
A stable strongly parabolic Higgs bundle $(W,\Phi)$ over $S^2$ with singularity divisor $\mathfrak{D}=p_1+p_2+p_3$ gives rise to a real representation, that is $[(W,\Phi)]\in \mathcal{M}^s_D(\R)$, if and only if $(W,\Phi)=(\mathcal{O}(-1)\oplus \mathcal{O}\oplus \mathcal{O}(1),\Phi^Q)$ from Lemma~\ref{lem:cyclicHiggs} or $(W,\Phi)=(\mathcal{O}^{\oplus 3}, \Phi)$ lies in the family \eqref{eq:Higgsgen} of Lemma~\ref{lem:trivialHiggs} for $y=-1/2$.
\end{Lem}
\begin{proof}
Lemmas~\ref{lem:bundletype}, \ref{lem:cyclicHiggs} and \ref{lem:trivialHiggs} show that the stable Dolbeault space $\mathcal{M}_D^s$ is parametrized by the families $(\mathcal{O}(-1)\oplus \mathcal{O}\oplus \mathcal{O}(1),\Phi^Q)$ and $(\mathcal{O}^{\oplus 3}, \Phi)$ 
given by \eqref{eq:Higgsgen}, \eqref{eq:Higgstriv+}, \eqref{eq:Higgstriv-}, and \eqref{eq:Higgstriv0}. As we have already shown in Lemma~\ref{lem:3}, the family $(\mathcal{O}(-1)\oplus \mathcal{O}\oplus \mathcal{O}(1),\Phi^Q)$ corresponds to real representations. 
Thus, it remains to investigate the families from Lemma~\ref{lem:trivialHiggs}. From Theorem~\ref{thm:charvar} (iii), we know that a representation $[\rho]\in\mathcal{M}_B^s$ corresponding to a Higgs bundle $[(W,\Phi)]\in\mathcal{M}_D^{s}$ is real, if an only if  the representation $\rho$ is equivalent to  its complex conjugate representation $\bar{\rho}$. By  Proposition~\ref{pro:reality}, the latter is equivalent to the existence of a Higgs bundle isomorphism $(W,\Phi)\cong (W^*, \Phi^*)$. Choosing the standard inner product on $\mathcal{O}^{\oplus 3}$, which provides an isomorphism $(W^*, \Phi^*)\cong (W,\Phi^t)$, we thus have to check for which of the families in Lemma~\ref{lem:trivialHiggs} there exists $g\in {\bf SL}(3,\C)$ such that $g\Phi=\Phi^t g$.  This last condition is equivalent to the system of linear equations $g\Res_{p_k}\Phi=\Res_{p_k} \Phi^t g$, $k=1,2$, for the matrix  $g=(g_{ij})$.  For the first family \eqref{eq:Higgsgen}, the equation $g\Res_{p_1}\Phi=\Res_{p_1}\Phi^t g$ implies 
\[
g=\begin{pmatrix} 0& a& a\\ a&0& b\\a&b &c\end{pmatrix}.
\]
Evaluating the second equation $g\Res_{p_2}\Phi=\Res_{p_2}\Phi^t g$, we obtain $b=-\tfrac{a}{2}$, $c=0$, and $y=-1/2$. Then $1=\det g=-a^3$ determines  $g\in {\bf SL}_2(\C)$ up to a cube root of unity. 

The considerations for the remaining three families in Lemma~\ref{lem:trivialHiggs} are analogous. In each of these cases one quickly sees that $\det g=0$, and therefore none of these families give rise to real representations.
\end{proof}
\begin{The}\label{thm:last}
Consider the 2-sphere $S^2$ with singularity divisor $\frak{D}=p_1+p_2+p_3$ and $\Sigma=S^2\setminus\{p_1,p_2,p_3\}$ the thrice-punctured sphere. Let 
$\mathcal{M}^{ps}_B(\R)=\mathcal{C}_1\cup \mathcal{C}_2$ be the Betti space of completely reducible $ {\bf SL}_3(\R)$ representations of $\pi_1(\Sigma)$. Then we have:    
\begin{enumerate}[(i)]
\item 
The map
\[
 H^0(K^3\mathcal{O}(2\frak{D}))\to \mathcal{C}_2\colon Q\mapsto [\rho^Q]
\]
is a homeomorphism, where $\rho^Q$ is 
the representation corresponding to the Higgs bundle  $(\mathcal{O}(-1)\oplus \mathcal{O}\oplus \mathcal{O}(1),\Phi^Q)$ from Lemma~\ref{lem:cyclicHiggs}. 
\item
The map
\[
\C\to \mathcal{C}_1 \colon x\mapsto [\rho^x]\,,
\]
assigning the family of Higgs bundles $(\mathcal{O}^{\oplus 3},\Phi^{(x,-1/2)})$  from Lemma~\ref{lem:trivialHiggs}, \eqref{eq:Higgsgen}
to the corresponding representation $\rho^x$, is a homeomorphism.
\end{enumerate}
\end{The}
\begin{proof}
From \cite{S}, we know that the correspondence 
\[
\mathcal{M}_{D}^{ps}  \cong \mathcal{M}_B^{ps}  \colon  [(W,\Phi)]\leftrightarrow   [\rho]
\]
between polystable parabolic Higgs bundles and completely reducible surface group representations is a bijection. Furthermore, \cite{KW} shows that along smooth families of parabolic Higgs bundles $(W,\Phi)$,  the harmonic metric $h$ depends real analytically on the parabolic Higgs bundle. Therefore, the corresponding flat connection $d=D^h+\Phi+\Phi^{\dagger_h}$, and thus its monodromy representation $\rho$, also depends real analytically on the Higgs bundle. We have seen, that the stable locus $\mathcal{M}^s_B$ of the Betti space is smooth and contains the Hitchin component $\mathcal{C}_2$. Since the Higgs bundle  $(\mathcal{O}(-1)\oplus \mathcal{O}\oplus \mathcal{O}(1),\Phi^Q)$ for $Q\equiv 0$ maps to the uniformization representation, the first map  of Theorem~\ref{thm:last} is well-defined, injective, and real analytic. For the second map, we note that the real analytic assignment  $x\mapsto [\rho^x]$ for $x\neq 0$ extends continuously into $x=0$  with image the trivial representation. Therefore, the second map of Theorem~\ref{thm:last} is well-defined, injective,  and continuous. Since $\mathcal{C}_1,\, \mathcal{C}_2\subset \mathcal{M}_{B}^{ps}$ are the connected components,  Lemma~\ref{lem:realDol} implies that both maps of Theorem~\ref{thm:last} are bijective. 
Applying invariance of domain, we thus can conclude that the maps are homeomorphisms. 
\end{proof}
\begin{Rem}
Using Lemmas~\ref{lem:cyclicHiggs} and \ref{lem:trivialHiggs}, one can show---with some additional  work---that $\mathcal{M}^{s}_D$ is a 2-dimensional complex manifold.
Combined with \cite{KW}, one could then conclude that the correspondence $\mathcal{M}_{D}^{s}  \cong \mathcal{M}_B^{s}$ is a real analytic diffeomorphism extending to a homeomorphism $\mathcal{M}_{D}^{ps}  \cong \mathcal{M}_B^{ps}$. Lemma~\ref{lem:realDol} characterizes the smooth real locus $\mathcal{M}^{s}_D(\R)$ as the two connected components given by the smooth families $(\mathcal{O}(-1)\oplus \mathcal{O}\oplus \mathcal{O}(1),\Phi^Q)$  and $(\mathcal{O}^{\oplus 3},\Phi^{(x,-1/2)})$, $x\neq 0$. These then correspond, via restriction of the real analytic diffeomorphism $\mathcal{M}_{D}^{s}\cong\mathcal{M}_B^{s}$, to $\mathcal{C}_2$ and $\mathcal{C}_1\setminus \{[\mathrm{I}]\}$.
\end{Rem}

 \bigskip
 \noindent \footnotesize \textsc{Beijing Institute of Mathematical Sciences and Applications}\\
\emph{E-mail address:}  \verb|sheller@bimsa.cn|
 
 \bigskip
 \noindent \footnotesize \textsc{Department of Mathematics, Washington University}\\
\emph{E-mail address:}  \verb|ouyang@math.wustl.edu|

\bigskip
\noindent \footnotesize \textsc{Department of Mathematics and Statistics, University of Massachusetts Amherst}\\
\emph{E-mail address:}  \verb|pedit@math.umass.edu|

\end{document}